\documentclass[11pt]{article}
\usepackage[a4paper,margin=2.5cm,top=0.8in]{geometry}

\usepackage{graphicx} \usepackage{amsmath,amsthm,amssymb,amsfonts,amscd}
\usepackage{hyperref,cleveref}
\usepackage{xcolor}
\usepackage{enumitem}
\usepackage{tikz}
\usepackage{yhmath} \usetikzlibrary{calc}

\newtheorem{theorem}{Theorem}[section]
\newtheorem{lemma}[theorem]{Lemma}

\newtheorem{claim}{Claim}[theorem]
\newtheorem{corollary}[theorem]{Corollary}
\newtheorem{proposition}[theorem]{Proposition}

\newtheorem*{claim*}{Claim}

\theoremstyle{definition}
\newtheorem{definition}[theorem]{Definition}
\newtheorem{remark}[theorem]{Remark}
\newtheorem{example}[theorem]{Example}
\newtheorem{question}[theorem]{Question}
\newtheorem{convention}[theorem]{Convention}

\newcommand{\symdiff}{\triangle}

\newcommand{\pf}{\mathrm{pf}}

\newcommand{\cB}{\mathcal{B}}

\newcommand{\cQ}{\mathcal{Q}}

\newcommand{\bfA}{\mathbf{A}}
\newcommand{\bfC}{\mathbf{C}}

\newcommand{\bfI}{\mathbf{I}}

\newcommand{\bfM}{\mathbf{M}}

\newcommand{\bfv}{\mathbf{v}}

\newcommand{\bB}{\mathbb{B}}
\newcommand{\bC}{\mathbb{C}}
\newcommand{\bF}{\mathbb{F}}
\newcommand{\bG}{\mathbb{G}}
\newcommand{\bH}{\mathbb{H}}
\newcommand{\bR}{\mathbb{R}}
\newcommand{\bZ}{\mathbb{Z}}

\usepackage{authblk}
\title{Pseudo-orientable ribbon graphs: \\ Matrix--Quasi-tree Theorem and log-concavity}

\author[1]{Changxin Ding\thanks{\href{mailto:dcx.math@outlook.com}{dcx.math@outlook.com}. Changxin Ding was supported by the AMS-Simons Travel Grant.}}
\author[1]{Donggyu Kim\thanks{\href{mailto:donggyu@gatech.edu}{donggyu.math@gmail.com}. Donggyu Kim was supported by the AMS-Simons Travel Grant.}}
\affil[1]{School of Mathematics, Georgia Institute of Technology, USA}

\date{}   

\begin{document}

\vspace{-2em}

\maketitle

\vspace{-3em}

\begin{abstract}

One of the most important classes of even $\Delta$-matroids arises from orientable ribbon graphs, which play a role analogous to that of graphic matroids in matroid theory. Motivated by a natural correspondence between strong $\Delta$-matroids and even $\Delta$-matroids due to Geelen and Murota, we characterize the class of strong $\Delta$-matroids that correspond to orientable ribbon-graphic $\Delta$-matroids. These are precisely the $\Delta$-matroids associated with what we call \emph{pseudo-orientable} ribbon graphs. Moreover, we present a geometric construction that transforms a pseudo-orientable ribbon graph into an orientable ribbon graph, thereby realizing this correspondence.
    
As consequences, we obtain the Matrix--Quasi-tree Theorem, the Hurwitz stability of quasi-tree generating polynomials, and the log-concavity of the sequence counting quasi-trees of size $2i-1$ or $2i$ for pseudo-orientable ribbon graphs.
To establish this log-concavity, we generalize Stanley's log-concavity theorem for regular matroids to regular $\Delta$-matroids.
Finally, we exhibit an infinite family of non-pseudo-orientable ribbon graphs that fail to satisfy the Matrix--Quasi-tree theorem and Hurwitz stability.
\end{abstract}

\section{Introduction}

A \emph{ribbon graph} is a graph with additional topological data, which can be viewed as a graph cellularly embedded in a (possibly, non-orientable) closed surface. When the surface is a plane, the ribbon graph is a plane graph. A \emph{quasi-tree} is the edge set of a spanning ribbon subgraph with exactly one boundary component; for a connected plane graph, the quasi-trees are exactly the spanning trees.
Orientable ribbon graphs sit at a particularly tractable intersection of topological graph theory and linear algebra: the quasi-trees of an orientable ribbon graph form an even $\Delta$-matroid that admits a principally unimodular (PU) skew-symmetric matrix representation. 
It implies several interesting properties of orientable ribbon graphs, such as the Matrix--Quasi-tree Theorem, the Hurwitz stability of quasi-tree generating polynomials, and the canonical Jacobian group action on quasi-trees~\cite{MMN2025b,BDK2026}.
In contrast, general ribbon graphs $\bG$ (possibly, non-orientable) are less understood largely because the associated (strong) $\Delta$-matroids $D(\bG)$ might not admit well-behaved matrix representations.

Our main contribution is the definition and study of a subclass of ribbon graphs, called \emph{pseudo-orientable} ribbon graphs, which contains all orientable ribbon graphs. This class enjoys several properties analogous to those of orientable ribbon graphs:
\begin{itemize}
\item A Matrix--Quasi-tree Theorem holds (Theorem~\ref{thm-intro: matrix--quasi-tree for pseudo-orientable}).
\item The class is closed under taking minors (Proposition~\ref{prop: partial duality preserves pseudo-orientability}).
\item The quasi-tree generating polynomial is Hurwitz stable (Theorem~\ref{thm-intro: stability}).
\item An appropriate sequence derived from the numbers of quasi-trees of different sizes is ultra-log-concave (Theorem~\ref{thm-intro: log-concavity}), which is a new result even for orientable ribbon graphs.
\end{itemize}

Our motivation for introducing pseudo-orientability is twofold. First, although even $\Delta$-matroids are a subclass of strong $\Delta$-matroids, there is a natural bijection (Corollary~\ref{cor: BD bijection}):
\begin{align*}
\{\text{strong $\Delta$-matroids on $E$}\}
&\to
\Big\{
\hspace{-1.5mm}
\begin{array}{c}
\text{even $\Delta$-matroids on } E\cup \{\widehat{e}\} \\[-2pt]
\text{with even-sized bases}
\end{array}
\hspace{-1.5mm}
\Big\} \\
D \hspace{2.0cm} &\mapsto \hspace{2.5cm} \widehat{D}.
\end{align*}
Since the element $\widehat{e}$ may be regarded as an auxiliary element added to the ground set, we call $\widehat{D}$ the \emph{lift} of $D$. To the best of our knowledge, this result first appeared in \cite{Murota2021}, where Murota proved a more general statement and attributed it to Geelen. Recently, this lift was rediscovered by Calvert, Dermenjian, Fink, and Smith in \cite{CDFS2025} where the authors study Coxeter matroids of various types; see also~\cite{GKL2025}. In the literature, $\Delta$-matroids and even $\Delta$-matroids are often viewed as type B and type D, respectively; see \cite{BGW2003}. From the perspective of \cite{CDFS2025}, Coxeter matroids of type B are precisely strong $\Delta$-matroids.

Second, we observe that for representable $\Delta$-matroids, the lift has been studied implicitly in terms of matrices in \cite{vGM2019,CCLV2025}. At the matrix level, the lift can be interpreted as passing from the matrix $\bfA + \bfv\bfv^T$ to the  matrix $\begingroup
\setlength{\arraycolsep}{2pt}
\renewcommand{\arraystretch}{0.85}
\begin{pmatrix}
    \bfA & \bfv \\
    -\bfv^T & 0 \\
\end{pmatrix}
\endgroup$, where $\bfA$ is skew-symmetric and $\bfv$ is a column vector; for more details, see \S\ref{sec: strong delta-matroids}.

Our goal is to understand this lift at the level of ribbon graphs. Note that any ribbon graph $\bG$ gives a strong $\Delta$-matroid $D(\bG)$ (Cor.~\ref{cor: ribbon is strong}), and $D(\bG)$ is even if and only if $\bG$ is orientable. To this end, we characterize the ribbon graphs $\bG$ for which the lift of the $\Delta$-matroid $D(\bG)$ arises from an orientable ribbon graph. Our first main result, Theorem~\ref{thm-intro: pseudo-orientable and orientable}, shows that pseudo-orientability (Definitions~\ref{def: pseudo-orientable bouquet} and~\ref{def: pseudo-orientable ribbon graph}) provides the correct characterization.

\begin{theorem}\label{thm-intro: pseudo-orientable and orientable}
    A ribbon graph $\bG$ admits an orientable ribbon graph $\bH$ with $D(\bH)$ isomorphic to the lift of $D(\bG)$ if and only if $\bG$ is pseudo-orientable up to the notion of $2$-isomorphism for ribbon graphs given in~\cite{MO2021}.
    Moreover, when $\bG$ is pseudo-orientable, one can construct such an $\bH$ explicitly.
\end{theorem}

For the construction of $\bH$ in the theorem, we define a geometric operation $\widehat{\bG}$ on a pseudo-orientable ribbon graph $\bG$, called an \emph{adjustment} of $\bG$, such that $D(\widehat{\bG})$ is the lift of $D(\bG)$. In this sense, the adjustment $\widehat{\bG}$ realizes the lift of strong $\Delta$-matroids.

\begin{figure}[h!]
    \centering
    \begin{tikzpicture}
\begin{scope}[scale=0.95]

\fill[draw=white, line width=1pt, fill=blue, even odd rule]
                (0,0.6) circle (0.5-0.12)  
                (0,0.6) circle (0.5+0.12);
            \node at (110:1.33) {$2$};

\fill[draw=white, line width=1pt, fill=blue, even odd rule]
                (0,-0.6) circle (0.5-0.12)  
                (0,-0.6) circle (0.5+0.12);
            \node at (-110:1.33) {$3$};

\fill[draw=white, line width=2.9pt, fill=red, even odd rule]
                (0.8,0.12) circle (1)  
                (0.8,-0.12) circle (1);
            \fill[fill=red, even odd rule]
                (0.8,0.12) circle (1)  
                (0.8,-0.12) circle (1);
            \draw[line width=6pt, red] (0.8,1) arc (90:270:1);
            \node at (-10:1.52) {$1$};

            \draw[line width=2pt, draw=black, line width=2pt, fill=white] (0,0) circle (0.6);
            \draw[dotted, line width=2pt] (-0.6,0) -- (0.6,0);

            \node at (-1.0,0) {$\bG$};
        \end{scope}
\begin{scope}[xshift=6.4cm, scale=0.95]
\fill[draw=white, line width=1pt, fill=blue, even odd rule]
                (0,0.6) circle (0.5-0.12)  
                (0,0.6) circle (0.5+0.12);
            \node at (90:1.37) {$2$};

\fill[draw=white, line width=1pt, fill=blue, even odd rule]
                (0,-0.6) circle (0.5-0.12)  
                (0,-0.6) circle (0.5+0.12);
            \node at (-110:1.33) {$3$};

\fill[draw=white, line width=1.5pt, fill=blue, even odd rule]
                (0.8,0) circle (1-0.12)  
                (0.8,0) circle (1+0.12);
            \node at (-10:1.52) {$1$};

\fill[draw=white, line width=1.5pt, fill=blue, even odd rule]
                (0,0.8) circle (1-0.12)  
                (0,0.8) circle (1+0.12);
            \node at (57:1.85) {$4$};
            
            \draw[line width=2pt, draw=black, line width=2pt, fill=white] (0,0) circle (0.6);
            \draw[dotted, line width=2pt] (-0.6,0) -- (0.6,0);

            \node at (-1.25,0) {$\widehat{\bG}$};
        \end{scope}
\begin{scope}[yshift=-2.60cm]
\node at (0,0) {
                $
                \bfA + \bfv \bfv^T =
                \begin{pmatrix}
                    1 & 1 & 1 \\
                    -1 & 0 & 0 \\
                    -1 & 0 & 0 \\
                \end{pmatrix}
                $
            };
        \end{scope}
\begin{scope}[yshift=-2.60cm, xshift=6.4cm]
\node at (0,0) {
                $
                \begingroup
                \setlength{\arraycolsep}{2pt}
                \renewcommand{\arraystretch}{0.85}
                \begin{pmatrix}
                    \bfA & \bfv \\
                    -\bfv^T & 0 \\
                \end{pmatrix}
                \endgroup
                =
                \left(
                \begin{array}{ccc|c}
                    0 & 1 & 1 & 1 \\
                    -1 & 0 & 0 & 0 \\
                    -1 & 0 & 0 & 0 \\ \hline
                    -1 & 0 & 0 & 0 \\
                \end{array}
                \right)
                $
            };
        \end{scope}
\begin{scope}[yshift=-4.0cm]
            \node at (0,0) {
                $
                D = ([3],\{\emptyset, 1, 12, 13\})
                $
            };
        \end{scope}
\begin{scope}[yshift=-4.0cm, xshift=6.4cm]
            \node at (0,0) {
                $
                \widehat{D} = ([4],\{\emptyset, 14, 12, 13\})
                $
            };
        \end{scope}
    \end{tikzpicture}
    \caption{A pseudo-orientable ribbon graph $\bG$ and its adjustment $\widehat{\bG}$, together with the corresponding matrices and $\Delta$-matroids.
        $\widehat{\bG}$ is obtained by flipping the upper half-circle of the vertex of $\bG$ and adding a new orientable loop, denoted by $4$, which connects the two boundary points of the intersection of the half-circles.
        The matrix $\bfA$ encodes the interlacements along the loops in $\widehat{\bG} \setminus 4$, while the vector $\bfv$ records the non-orientable loops in the original $\bG$.
        The associated $\Delta$-matroids $D$ and $\widehat{D}$ encode the quasi-trees of $\bG$ and $\widehat{\bG}$, respectively.
}
    \label{fig: intro}
\end{figure}

Figure~\ref{fig: intro} illustrates how a pseudo-orientable ribbon graph $\bG$ is adjusted to the corresponding orientable $\widehat{\bG}$, together with the associated interlacing matrices and $\Delta$-matroids.

Every pseudo-orientable ribbon graph $\bG$ can be represented by the matrix $\bfA + \bfv\bfv^T$. 
Following this approach, we obtain the Matrix--Quasi-tree Theorem. 

\begin{theorem}\label{thm-intro: matrix--quasi-tree for pseudo-orientable}
    Let $\bG$ be a pseudo-orientable ribbon graph and $Q$ be a quasi-tree.
    There is an integral PU matrix $\bfM$ such that 
    \[
        \det(\bfM[Q \symdiff X]) = 
        \begin{cases}
            1 & \text{if $X$ is a quasi-tree of $\bG$,}\\
            0 & \text{otherwise}.
        \end{cases}
        \tag{$*$}
        \label{eq: matrix--quasi-tree for pseudo-orientable}
    \]
    In particular, $\det(\bfI + \bfM)$ equals the number of quasi-trees of $\bG$.
\end{theorem}

The Matrix--Quasi-tree Theorem for \emph{orientable} ribbon graphs was implicit in several works and stated explicitly by Merino, Moffatt, and Noble~\cite{MMN2025b}; also see references therein.
Deng, Jin, and Yan~\cite{DJY2024} showed the Matrix--Quasi-tree Theorem for bouquets with exactly one non-orientable loop. Our result generalizes both of their results because the class of pseudo-orientable ribbon graphs contains orientable ones and bouquets with exactly one non-orientable loop.

Moreover, we find
an infinite family of non-pseudo-orientable ribbon graphs that do not admit any matrix representation satisfying the property~\eqref{eq: matrix--quasi-tree for pseudo-orientable}.

The quasi-tree generating polynomials of orientable ribbon graphs are known to be \emph{Hurwitz stable} \cite{MMN2025b}. We generalize this result to pseudo-orientable ribbon graphs. For the infinite family of non-pseudo-orientable ribbon graphs mentioned above, the corresponding quasi-tree generating polynomials are not Hurwitz stable.
\begin{theorem}\label{thm-intro: stability}
    The quasi-tree generating polynomial 
    \[p_{\bG} := \sum_{Q:\text{quasi-trees}} \prod_{e\in Q} x_e\]
    of a pseudo-orientable ribbon graph $\bG$ is Hurwitz stable.
\end{theorem}

Hurwitz stability is closely related to \emph{log-concavity}. We show that, for a pseudo-orientable ribbon graph, the sequence counting quasi-trees of size $2i-1$ or $2i$ is ultra-log-concave. This result is new even for orientable ribbon graphs. 

\begin{theorem}\label{thm-intro: log-concavity}
    Let $\bG$ be a pseudo-orientable ribbon graph and $Q \subseteq E(\bG)$.
    Let 
    \[
        q^Q_{i} := \text{the number of quasi-trees $X$ with $|Q\symdiff X| = 2i-1$ or $2i$}.
    \]
    Then the sequence $(q^Q_i)_{i\ge 0}$ is ultra-log-concave with no internal zeros.
    In particular, the sequence $(q_i := q^\emptyset_i)_{i\ge 0}$ counting quasi-trees of size $2i-1$ or $2i$ is ultra-log-concave with no internal zeros.
\end{theorem}

To prove this theorem, we first establish a log-concavity result for regular (even) $\Delta$-matroids (Theorem~\ref{thm: Stanley for regular delta-matroids}), which generalizes Stanley's log-concavity theorem for regular matroids \cite{Stanley1981}. The proof uses Hurwitz stability and follows the method in \cite{Yan2023}. We then pull back the data from orientable ribbon graphs to pseudo-orientable ribbon graphs.
It is easy to see that, for almost all orientable ribbon graphs, the sequence counting quasi-trees of size 
$i$ contains internal zeros. Hence, merging the numbers of quasi-trees of sizes $2i-1$ or $2i$ is a natural choice in Theorem~\ref{thm-intro: log-concavity}.
Ultra-log-concavity does not extend to all ribbon graphs, and in \S\ref{sec: failure of ultra-log-concavity}, we leave open the question of whether log-concavity, rather than ultra-log-concavity, holds for all ribbon graphs.

The paper is organized as follows.
In \S\ref{sec: preliminaries}, we review basic definitions and properties of ribbon graphs and $\Delta$-matroids. We also discuss the natural lift of strong $\Delta$-matroids $D$ to even $\Delta$-matroids $\widehat{D}$ and the analogous notion for matrices.
In \S\ref{sec: pseudo-orientable ribbon graphs}, we define pseudo-orientability and the adjustment operation.
We then prove Theorem~\ref{thm-intro: pseudo-orientable and orientable} and Proposition~\ref{prop: partial duality preserves pseudo-orientability}.  
In \S\ref{sec: stability and log-concavity}, we prove the results on Hurwitz stability and log-concavity (Theorems~\ref{thm-intro: stability} and~\ref{thm-intro: log-concavity}).
In \S\ref{sec: non-pseudo-orientable ribbon graphs}, we present an infinite family of non-pseudo-orientable ribbon graphs that cannot be represented by any matrix with the property \eqref{eq: matrix--quasi-tree for pseudo-orientable} and whose quasi-tree generating polynomials are not Hurwitz stable.

\section{Preliminaries}\label{sec: preliminaries}

We will review two preliminary topics: \S\ref{sec: ribbon graphs} ribbon graphs and  \S\ref{sec: delta-matroids} $\Delta$-matroids.

\subsection{Ribbon graphs}\label{sec: ribbon graphs}

Ribbon graphs can be understood intuitively as an $\epsilon$-thickening of a graph cellularly embedded in a closed (possibly non-orientable) surface, for sufficiently small $\epsilon > 0$. They are also called \emph{fat graphs} or \emph{embedded graphs}, and 
can be described combinatorially via \emph{signed rotation systems}.
Ribbon graphs play a fundamental role in the study of knot theory and the theory of vertex-minors of graphs; see~\cite{CMNR2019a,EM2013,Moffatt2019} and the references therein.

We use a geometric definition from \cite{Chmutov2009}. 

\begin{definition}
A \emph{ribbon graph} $\bG$ is a surface (possibly non-orientable) with boundary, represented as
the union of two sets of closed topological discs called vertices $V(\bG)$ and edges $E(\bG)$, satisfying the
following conditions:

    \begin{itemize}
        \item these vertices and edges intersect by disjoint line segments;
        \item each such line segment lies on the boundary of precisely one vertex and precisely one edge;
        \item every edge contains exactly two such line segments, which are called the \emph{ends} of the edge.
    \end{itemize}
\end{definition}

Two ribbon graphs $\bG$ and $\bH$ are \emph{isomorphic} if there are bijections $\iota_V: V(\bG) \to V(\bH)$ and $\iota_E: E(\bG) \to E(\bH)$ along with homeomorphisms $v \to \iota_V(v)$ and $e \to \iota_E(e)$ for $v\in V(\bG)$ and $e\in E(\bG)$ that extend to a homeomorphism $\bG \to \bH$.

\begin{convention}
    Throughout the paper, we work with ``edge-labeled'' ribbon graphs.
    Formally, a ribbon graph on a finite set $E$ is a ribbon graph $\bG$ with $E(\bG) = E$.
    In particular, vertices are regarded as unlabeled, and we identify two ribbon graphs on $E$ if there is a ribbon graph isomorphism that induces the identity on $E$. See Proposition~\ref{prop: partial duality} for an example where this convention is applied. 
\end{convention}

A \emph{ribbon subgraph} of a ribbon graph $\bG$ is a ribbon graph $\bH$ with $V(\bH) \subseteq V(\bG)$ and $E(\bH) \subseteq E(\bG)$.
When the vertex sets coincide, we say $\bH$ is \emph{spanning}.
A \emph{quasi-tree} of a connected ribbon graph $\bG$ is a spanning ribbon subgraph with exactly one boundary component.\footnote{In some literature, a quasi-tree is not required to be spanning.}
For a (not necessarily connected) ribbon graph $\bG$, a \emph{quasi-tree} is defined as a union of quasi-trees of the connected components of $\bG$.
We also call an edge subset $F$ of $E(\bG)$ a \emph{quasi-tree} if the corresponding spanning ribbon subgraph $(V(\bG), F)$ is a quasi-tree.
We denote by $\cQ(\bG)$ the set of quasi-trees, as subsets of $E(\bG)$, of $\bG$.

\begin{definition}
    A \emph{bouquet} is a ribbon graph with a single vertex.
\end{definition}

We often call an edge of a bouquet a \emph{loop}.
Two loops \emph{interlace} if their ends appear alternatively on the boundary of the unique vertex $v$ of the bouquet.
A loop $e$ is said to be \emph{orientable} if the surface $v \cup e$ is orientable; equivalently, it is homeomorphic to an annulus.
Otherwise, $e$ is said to be \emph{non-orientable} in which case $v \cup e$ is homeomorphic to a M\"obius band.
For example, in Figure~\ref{fig: signed chord diagram}, the given bouquet has two orientable loops (colored blue) that do not interlace, and one non-orientable loop (colored red) that interlaces with the orientable loops.

For convenience, we will often identify a bouquet with a signed chord diagram:

\begin{definition}
    A \emph{signed chord diagram} is a circle with a finite number of chords, each of which is assigned $0$ or $1$. 
\end{definition}

Given a bouquet $\bB$, taking the boundary of the vertex gives a circle, and each edge corresponds to a chord connecting the two endpoints on the circle. We assign $0$ to each orientable loop and $1$ to each non-orientable loop.
This gives a signed chord diagram corresponding to $\bB$. See Figure~\ref{fig: signed chord diagram} for an example.

\begin{figure}[h!]
    \centering
    \begin{tikzpicture}[scale=1]
        \begin{scope}

\fill[draw=white, line width=1pt, fill=blue, even odd rule]
                (0,0.6) circle (0.5-0.12)  
                (0,0.6) circle (0.5+0.12);
            \node at (110:1.33) {$2$};

\fill[draw=white, line width=1pt, fill=blue, even odd rule]
                (0,-0.6) circle (0.5-0.12)  
                (0,-0.6) circle (0.5+0.12);
            \node at (-110:1.33) {$3$};

\fill[draw=white, line width=2.9pt, fill=red, even odd rule]
                (0.8,0.12) circle (1)  
                (0.8,-0.12) circle (1);
            \fill[fill=red, even odd rule]
                (0.8,0.12) circle (1)  
                (0.8,-0.12) circle (1);
            \draw[line width=6pt, red] (0.8,1) arc (90:270:1);
            \node at (-10:1.52) {$1$};

            \draw[line width=2pt, draw=black, line width=2pt, fill=white] (0,0) circle (0.6);

\end{scope}
        \begin{scope}[xshift=6cm]
            \draw[line width=3pt, blue] (40:2) -- (140:2);
            \draw[line width=3pt, blue] (-40:2) -- (-140:2);
            \draw[line width=3pt, red] (90:2) -- (-90:2);

            \node at (130:1.3) {$2$};
            \node at (-130:1.3) {$3$};
            \node at (0:0.2) {$1$};

            \draw[line width=2.5pt] (0,0) circle (2);

\end{scope}
    \end{tikzpicture}
    \caption{A bouquet (left) with three loops and the corresponding signed chord diagram (right). The edges and chords colored blue indicate orientable loops (assigned $0$), and the ones colored red indicate non-orientable loops (assigned $1$).}
    \label{fig: signed chord diagram}
\end{figure}
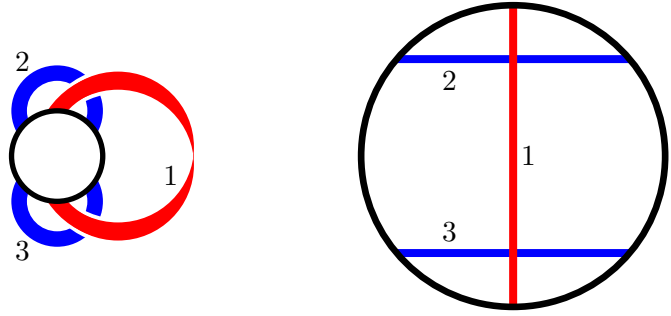

Chmutov~\cite{Chmutov2009} introduced partial duality for ribbon graphs as a generalization of duality for plane graphs and, more broadly, duality for ribbon graphs.
The partial dualities preserve the number of quasi-trees (cf. Prop.~\ref{prop: partial dual and quasi-tree}). 

\begin{definition}[Partial duality]
    Let $\bG$ be a ribbon graph and $X \subseteq E(\bG)$.
    The \emph{partial dual} of $\bG$ at $X$ is a ribbon graph $\bG^X$ defined as follows:
    \begin{enumerate}
        \item Take $E(\bG^X) = E(\bG)$.
        \item Attach a disk along each boundary of the spanning ribbon subgraph $(V(\bG), X)$, such that the interiors of such disks are pairwise disjoint and also disjoint from the edge disks. 
        \item Take the disks in Step~2 as the vertices of $\bG^X$.
    \end{enumerate}
    For convenience, we denote $\bG^{e_1e_2\dots e_k} = \bG^X$ if $X=\{e_1,e_2,\dots,e_k\}$.
\end{definition}

When $X = E(\bG)$, the partial duality coincides with the geometric duality on cellularly embedded graphs because the boundaries of $\bG$ correspond to its faces.

Partial dualities have several properties. 

\begin{proposition}[\cite{Chmutov2009}]
    $\bG^X\setminus Y = (\bG\setminus Y)^X$ for any disjoint $X,Y\subseteq E(\bG)$.
\end{proposition}

\begin{proposition}[\cite{Chmutov2009}] \label{prop: partial duality}
    $(\bG^{X})^Y = \bG^{X\symdiff Y}$ for any $X,Y\subseteq E(\bG)$. In particular, $(\bG^X)^X = \bG$.
\end{proposition}

\begin{proposition}[\cite{Chmutov2009}]  \label{prop: partial duality preserves orientability}
    The number of connected components and (non-)orientability are invariant under partial duality.
\end{proposition}

By definition, $X$ is a quasi-tree of $\bG$ if and only if $\bG^X$ is a bouquet. More generally, we have the following result.

\begin{proposition}[\cite{CMNR2019a}] \label{prop: partial dual and quasi-tree}
    Let $\bG$ be a ribbon graph and $X\subseteq E(\bG)$.
    Then $Q\subseteq E(\bG)$ is a quasi-tree of $\bG$ if and only if $Q\symdiff X$ is a quasi-tree of $\bG^X$.
\end{proposition}

For a bouquet $\bB$, every non-orientable loop $e$ is a quasi-tree and every pair of orientable loops $f,f'$ that interlace forms a quasi-tree, which implies that $\bB^e$ and $\bB^{ff'}$ are also bouquets.
We call such operations on bouquets the \emph{elementary partial duality}.
See Figures~\ref{fig: elementary partial duality - nonorientable} and \ref{fig: elementary partial duality - orientable} for examples.

\begin{figure}[h!]
    \centering
    \begin{tikzpicture}[scale=1]
        \begin{scope}
            \draw[line width=3pt, red] (50:2) -- (-50:2);
            \draw[line width=3pt, blue] (90:2) -- (-130:2);

\draw[line width=3pt, blue] (20:2) to[bend left=40] (110:2);

            \draw[line width=3pt, blue] (-165:2) to[bend left=40] (-110:2);
\draw[line width=3pt, red] (-10:2) to[bend right=50] (-35:2);

            \draw[line width=3pt, red] (0:2) -- (180:2);

            \draw[line width=2.5pt] (0,0) circle (2);
            
            \draw[line width=1pt, dashed, ->] (5:2.15) arc (5:175:2.15);
            \draw[line width=1pt, dashed, <-] (-5:2.15) arc (-5:-175:2.15);
            \node at (17:2.5) {$S_1$};
            \node at (-20:2.5) {$S_2$};
            \node at (0:2.25) {$e$};

            \node at (0,-2.6) {$\bB$};
        \end{scope} 
        \begin{scope}[xshift=6cm]
            \draw[line width=3pt, blue] (50:2) -- (-130:2);
            \draw[line width=3pt, red] (90:2) -- (-50:2);

\draw[line width=3pt, blue] (20:2) to[bend left=40] (110:2);

            \draw[line width=3pt, blue] (-15:2) to[bend right=40] (-70:2);
\draw[line width=3pt, red] (-170:2) to[bend left=50] (-145:2);

            \draw[line width=3pt, red] (0:2) -- (180:2);

            \draw[line width=1pt, dashed, ->] (5:2.15) arc (5:175:2.15);
            \draw[line width=1pt, dashed, ->] (-5:2.15) arc (-5:-175:2.15);
            \node at (17:2.5) {$S_1$};
            \node at (-20:2.5) {$S_2$};
            \node at (0:2.25) {$e$};

            \draw[line width=2.5pt] (0,0) circle (2);
            
            \node at (0,-2.6) {$\bB^e$};
        \end{scope}
    \end{tikzpicture}
    \caption{
        The partial dual of a bouquet $\bB$ at a non-orientable loop $e$.
        The blue chords represent orientable loops, and the red chords represent non-orientable loops.
    }
    \label{fig: elementary partial duality - nonorientable}
\end{figure}
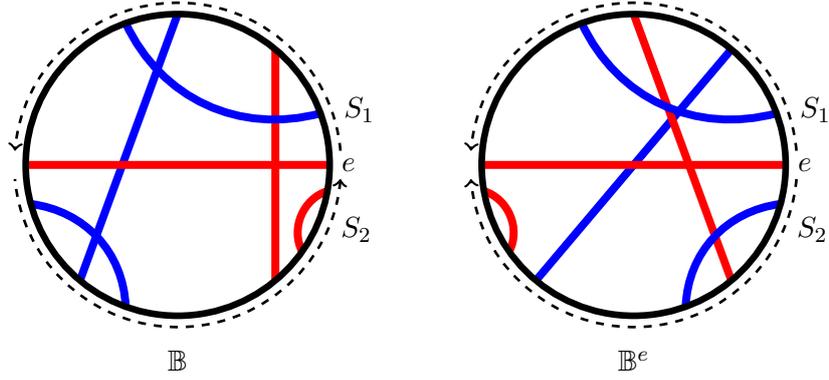

\begin{figure}[h!]
    \centering
    \begin{tikzpicture}[scale=1]
        \begin{scope}
            \draw[line width=3pt, blue] (0:2) -- (180:2);
            \draw[line width=3pt, blue] (90:2) -- (-90:2);

            \draw[line width=1pt, dashed, ->] (5:2.15) arc (5:85:2.15);
            \draw[line width=1pt, dashed, ->] (95:2.15) arc (95:175:2.15);
            \draw[line width=1pt, dashed, ->] (185:2.15) arc (185:265:2.15);
            \draw[line width=1pt, dashed, ->] (275:2.15) arc (275:355:2.15);

            \node at (45:2.5) {$S_1$};
            \node at (135:2.5) {$S_2$};
            \node at (225:2.5) {$S_3$};
            \node at (315:2.5) {$S_4$};

\draw[line width=3pt, blue] (20:2) to[bend left=40] (110:2);

\draw[line width=3pt, red] (165:2) to[bend right=60] (130:2);

\draw[line width=3pt, blue] (-165:2) to[bend left=40] (-110:2);

\draw[line width=3pt, red] (-140:2) to[bend left=30] (-40:2);

            \node at (0:2.24) {$f$};
            \node at (89:2.27) {$f'$};

            \draw[line width=2.5pt] (0,0) circle (2);

            \node at (0,-2.6) {$\bB$};
        \end{scope} 
        \begin{scope}[xshift=6cm]

            \draw[line width=3pt, blue] (0:2) -- (180:2);
            \draw[line width=3pt, blue] (90:2) -- (-90:2);

            \draw[line width=1pt, dashed, ->] (5:2.15) arc (5:85:2.15);
            \draw[line width=1pt, dashed, ->] (95:2.15) arc (95:175:2.15);
            \draw[line width=1pt, dashed, ->] (185:2.15) arc (185:265:2.15);
            \draw[line width=1pt, dashed, ->] (275:2.15) arc (275:355:2.15);

            \node at (45:2.5) {$S_3$};
            \node at (135:2.5) {$S_2$};
            \node at (225:2.5) {$S_1$};
            \node at (315:2.5) {$S_4$};

\draw[line width=3pt, blue] (20+180:2) to[bend right=40] (110:2);

\draw[line width=3pt, red] (165:2) to[bend right=60] (130:2);

\draw[line width=3pt, blue] (-165+180:2) to[bend left=40] (-110+180:2);

\draw[line width=3pt, red] (-140+180:2) to[bend right=30] (-40:2);

            \node at (0:2.24) {$f$};
            \node at (89:2.27) {$f'$};

            \draw[line width=2.5pt] (0,0) circle (2);
            
            \node at (0,-2.6) {$\bB^{\{f,f'\}}$};
        \end{scope}
    \end{tikzpicture}
    \caption{
        The partial dual of a bouquet $\bB$ at an interlacing pair of orientable loops $f$ and $f'$.
        The blue chords represent orientable loops, and the red chords represent non-orientable loops.
    }
    \label{fig: elementary partial duality - orientable}
\end{figure}
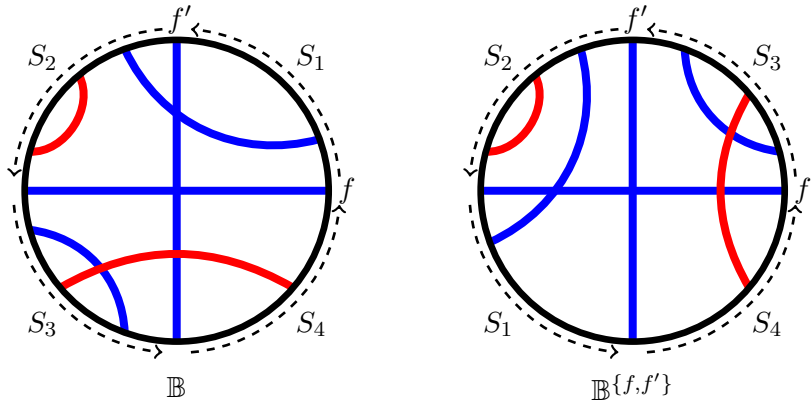

By the following proposition, in order to show a certain property on bouquets is closed under partial duality it suffices to check whether it is preserved under elementary partial duality.

\begin{proposition}[folklore]\label{prop: elementary partial duals}
    If two bouquets are partial duals of each other, then one can be obtained from the other by a sequence of elementary partial duals.
\end{proposition}

Our proof relies on the relation between ribbon graphs and $\Delta$-matroids; see \S\ref{sec: delta-matroids}. 

\begin{proof}
    Let $\bB$ and $\bB'$ be bouquets such that $\bB' = \bB^Q$ for some quasi-tree $Q$ of $\bB$.
    Denote $D := D(\bB)$. Then $\emptyset$, $Q$ are bases of $D$ and $D * Q = D(\bB^Q)$.
    By the basis exchange axiom for $\Delta$-matroids, we obtain a chain of bases 
    \[
        \emptyset = B_0 \subset B_1 \subset B_2 \subset \dots \subset B_k = Q
    \]
    such that $|B_i\setminus B_{i-1}| = 1$ or $2$ for each $i=1,2,\dots,k$.
    We may assume the chain is maximal.

    Note that $\bB_i := \bB^{B_i}$ is a bouquet for each $i=0,1,2,\dots,k$, and by Proposition~\ref{prop: partial duality}, $\bB_i$ is the partial dual of $\bB_{i-1}$ at $B_i\setminus B_{i-1}$ for each $i=1,2,\dots,k$. 
    
    Suppose $|B_i\setminus B_{i-1}| = 1$. Then the unique element in the set is a non-orientable loop in $\bB_{i-1}$. 

    Suppose $|B_i\setminus B_{i-1}| = 2$. Then the two elements in the set are orientable loops in $\bB_{i-1}$, since otherwise we can expand the chain of bases, contradicting the maximality. Because $B_i\setminus B_{i-1}$ is a quasi-tree, the two orientable loops must interlace.
    
    In each case, taking the partial dual $\bB_i$ of $\bB_{i-1}$ at $B_i\setminus B_{i-1}$ is elementary.
\end{proof}

We define minors of ribbon graphs. 

\begin{definition}[Ribbon graph minors] A ribbon graph $\bH$ is a \emph{minor} of a ribbon graph $\bG$ if $\bH$ can be obtained from $\bG$ by a sequence of edge deletions, vertex deletions, and partial duals.
\end{definition}

The following result describes how edge deletions change the set of quasi-trees. Recall that an edge $e$ of a ribbon graph $\bG$ is called a \emph{bridge} if its deletion increases the number of connected components; or equivalently, every quasi-tree of $\bG$ contains $e$.

\begin{proposition}[\cite{CMNR2019a}]
    \label{prop: edge deletion and quasi-tree}
    Let $\bG$ be a ribbon graph and $e \in E(\bG)$.
    \begin{enumerate}
        \item $\cQ(\bG \setminus e) = \{Q : e\notin Q\in \cQ(\bG)\}$ if $e$ is not a bridge.
        \item $\cQ(\bG \setminus e) = \{Q \setminus \{e\} : e\in Q\in \cQ(\bG)\}$ if $e$ is a bridge.
    \end{enumerate}
\end{proposition}

Lastly, we review partial Petrials of ribbon graphs, which might change the number of quasi-trees and orientability.

\begin{definition}[Partial Petrial]
    The \emph{\textup{(}partial\textup{)} Petrial} of a ribbon graph $\bG$ at $X\subseteq E(\bG)$ is a ribbon graph $\bG^{\tau(X)}$ obtained from $\bG$ by giving a half-twist to each edge in $X$.
\end{definition}

We denote $\bG^{\tau(e)}$ if $X = \{e\}$.
Note that, if $\bG$ is a bouquet, then $\bG^{e}$ is also a bouquet and the orientability of $e$ is toggled.

It is clear from the definition that $(\bG^{\tau(e)}) \setminus f = (\bG \setminus f)^{\tau(e)}$ for any distinct edges $e,f\in E(\bG)$.
The partial Petrial and the partial duality commute with each other when applied to distinct edges.
\begin{proposition}[\cite{EM2012}]
    $(\bG^{\tau(e)})^f = (\bG^f)^{\tau(e)}$ for any distinct edges $e,f\in E(\bG)$.
\end{proposition}
Remark that the partial Petrial and the partial duality do not commute when applied to the same edge; {\cite[Lemma~3.2]{EM2012}}.

\subsection{$\Delta$-matroids}\label{sec: delta-matroids}

A \emph{$\Delta$-matroid} is a set system $(E,\cB)$ with a finite set $E$ and a nonempty set $\cB$ of subsets, called \emph{bases}, of $E$ satisfying the \emph{basis exchange axiom} (for $\Delta$-matroids):
\begin{center}
    for any $B,B'\in \cB$ and any $x \in B\symdiff B'$, there is $y \in B\symdiff B'$ such that $B\symdiff\{x,y\} \in \cB$.
\end{center}
Then \emph{matroids} are exactly the $\Delta$-matroids all of whose bases have the same cardinality. A $\Delta$-matroid is \emph{even} if all the bases have the same parity.

We follow the basic terminology of~\cite[\S3]{Moffatt2019} for loops, coloops, twisting, and minors of $\Delta$-matroids. The twisting of a $\Delta$-matroid $D$ at a subset $X\subseteq E$ will be denoted by $D*X$.

In this subsection, we survey three important classes of $\Delta$-matroids: 
\S\ref{sec: representable delta-matroids} representable,
\S\ref{sec: ribbon-graphic delta-matroids} ribbon-graphic, and 
\S\ref{sec: strong delta-matroids} strong $\Delta$-matroids.
In particular, \S\ref{sec: strong delta-matroids} includes several new observations on strong $\Delta$-matroids, which are crucial for our proof of Theorem~\ref{thm-intro: matrix--quasi-tree for pseudo-orientable}.

\subsubsection{Representable $\Delta$-matroids}\label{sec: representable delta-matroids}

The representability of $\Delta$-matroids is defined through (skew-)symmetric matrices over a field, which generalize the representability of matroids.

A matrix $\bfA=(\bfA_{ij})$ is \emph{skew-symmetric} if $\bfA_{ij} = -\bfA_{ji}$ for all $i, j$ and $\bfA_{ii} = 0$ for all $i$.
We denote an identity matrix by $\bfI$. For a subset $X$ of the row indices of a matrix $\bfA$, we denote by $A[X]$ the principal submatrix of $\bfA$ indexed by $X$. We use the convention that $\det(\bfA[\emptyset])=1$.  

A square matrix with complex entries is \emph{principally unimodular} (in short, \emph{PU}) if all of its principal minors are $0$ or $\pm1$.

\begin{theorem}[\cite{Bouchet1988}]
    Let $\bfA$ be a symmetric or skew-symmetric matrix over a field $K$ with rows and columns indexed by $E$.
    Then, the set system $D(\bfA) = (E, \cB)$ with
    \[
        \cB := \{ I \subseteq E : \bfA[I] \text{ is nonsingular} \}
    \]
    is a $\Delta$-matroid.
\end{theorem}

Notice that if $\bfA$ is skew-symmetric, then $D(\bfA)$ is an even $\Delta$-matroid.
Geelen~{\cite[p.~27]{Geelen1996}} showed a partial converse; namely, $D(\bfA)$ is an even delta-matroid only if $\bfA$ is skew-symmetric or is a block matrix of the form 
$\begingroup
\setlength{\arraycolsep}{2pt}
\renewcommand{\arraystretch}{0.85}
\begin{pmatrix}
    0 & \bfC \\
    \bfC^T & 0 \\
\end{pmatrix}
\endgroup$
so that $D(\bfA) = D(\bfA')$ where $\bfA' = \begingroup
\setlength{\arraycolsep}{2pt}
\renewcommand{\arraystretch}{0.85}
\begin{pmatrix}
    0 & \bfC \\
    -\bfC^T & 0 \\
\end{pmatrix}
\endgroup$.

\begin{definition}\label{def: representable delta-matroid}
    Let $K$ be a field.
    A $\Delta$-matroid $D$ is \emph{representable over $K$} (or \emph{$K$-representable}) if $D = D(\bfA) \ast X$ 
for some symmetric or skew-symmetric matrix $\bfA$ over $K$ and some set $X$.
\end{definition}

A $\Delta$-matroid is \emph{binary} if it is representable over the field with two elements, and is \emph{regular} if it is even and representable over any field. In some literature, regular $\Delta$-matroids are defined as $\Delta$-matroids being represented by PU skew-symmetric matrices.
It is well known that these two notions of regularity are equivalent~{\cite[Thm.~4.13]{Geelen1996}}.

As we mentioned earlier, the representability of $\Delta$-matroids generalizes the representability of matroids.

\begin{proposition}[\cite{Bouchet1988}]
    Given a field $K$, a matroid is representable over $K$ in the usual sense if and only if it is representable over $K$ as a $\Delta$-matroid.
\end{proposition}

Moreover, we note that every $K$-representation of matroid, in the sense of $\Delta$-matroids, has the form $D(\bfA) \ast X$ where $X$ is a set and $\bfA$ is a matrix of the form 
$\begingroup
\setlength{\arraycolsep}{2pt}
\renewcommand{\arraystretch}{0.85}
\begin{pmatrix}
    0 & \bfC \\
    \pm\bfC^T & 0 \\
\end{pmatrix}
\endgroup$
with the first $|X|$ rows and columns indexed by $X$; see {\cite[p.~25]{Geelen1996}}

\subsubsection{Ribbon-graphic $\Delta$-matroids}\label{sec: ribbon-graphic delta-matroids}

Another natural class of $\Delta$-matroids arises from ribbon graphs.

\begin{theorem}[\cite{Bouchet1987b}]\label{thm: ribbon-graphic}
    Let $\bG$ be a ribbon graph.
    Then a set system $D(\bG) := (E(\bG),\cQ(\bG))$ is a $\Delta$-matroid.
    Moreover, $D(\bG)$ is even if and only if $\bG$ is orientable.
\end{theorem}

\begin{definition}
    A $\Delta$-matroid is \emph{ribbon-graphic} if it arises as in Theorem~\ref{thm: ribbon-graphic}.
\end{definition}

The class of graphic matroids and the class of ribbon-graphic $\Delta$-matroids are incomparable, and their intersection is exactly the matroids associated to planar graphs~\cite{CMNR2019a}. The graphic matroids of $K_5$ and $K_{3,3}$ are not ribbon-graphic by~\cite{GO2009}.\footnote{In~\cite{GO2009}, ribbon-graphic $\Delta$-matroids are named \emph{Eulerian} $\Delta$-matroids.}
On the other hand, the $\Delta$-matroid $([4],\{\emptyset,12,13,14,23,24,34,1234\})$ is ribbon-graphic realized by an orientable bouquet with four edges interlacing pairwise, but it is not a matroid even up to twisting.

As a consequence of Bouchet's results~\cite{Bouchet1988,Bouchet1987c}, orientable ribbon-graphic $\Delta$-matroids are regular.
Remarkably, Booth et al.~\cite{BBGS2000} presented another proof using cohomology of punctured surfaces, and the idea was used in \cite{BDK2026} to define the signed circuits of ribbon graphs, which leads to a canonical action of the Jacobian group of an orientable ribbon graph on the quasi-trees. 

\begin{theorem}[\cite{Bouchet1988,Bouchet1987c}]\label{thm: ribbon graphic is binary}
    Every ribbon-graphic $\Delta$-matroid is binary.
    Moreover, every orientable ribbon-graphic $\Delta$-matroid is regular.
\end{theorem}

Note that if $D$ is a binary $\Delta$-matroid with $\emptyset \in \cB(D)$, there is a unique binary symmetric matrix $A$ with $D=D(A)$.
Moreover, the binary representation of the ribbon-graphic $\Delta$-matroid associated with a bouquet can be constructed 
as follows.

\begin{proposition}[see the paragraph above Ex.~7.15 in~\cite{Moffatt2019}]
    \label{prop: binary interlacing matrix}
    Let $\bB$ be a bouquet and let $\bfM_2(\bB)$ be a binary matrix with rows and columns indexed by $E(\bB)$ and defined as
    \[
        \bfM_2(\bB)_{ef} = 
        \begin{cases}
            1   &   \text{if $e=f$ and the loop is non-orientable}, \\
            0   &   \text{if $e=f$ and the loop is orientable}, \\
            1   &   \text{if $e\ne f$ and the two loops interlace}, \\
            0   &   \text{otherwise}.
        \end{cases}
    \]
    Then $D(\bB) = D(\bfM_2(\bB))$.
\end{proposition}

\begin{definition}
    We call $\bfM_2(\bB)$ the \emph{\textup{(}binary\textup{)} interlacing matrix} of $\bB$. 
\end{definition}

The notion of ribbon graph minors and the notion of $\Delta$-matroid minors are compatible in the following sense. 

\begin{proposition}[\cite{CMNR2019a}] \label{prop: ribbon graph minor and delta-matroid minor}
    Let $\bG$ be a ribbon graph and $e$ be an edge.
    \begin{enumerate}
        \item $D(\bG^e) = D(\bG)*\{e\}$.
        \item $D(\bG \setminus e) = D(\bG) \setminus e$.
\end{enumerate}
\end{proposition}

Moffatt and Oh~\cite{MO2021} showed a ribbon-graphic analogue of Whitney's $2$-isomorphism theorem.

\begin{theorem}[\cite{MO2021}] \label{thm: 2-isomorphism}
    Let $\bG$ and $\bH$ be two ribbon graphs. Then $D(\bG) \cong D(\bH)$ if and only if $\bG$ can be obtained from $\bH$ by ribbon graph isomorphism, vertex joins, vertex cuts, and mutations.
\end{theorem}

We say that two ribbon graphs $\bG$ and $\bH$ are \emph{$2$-isomorphic} if $\bG$ can be obtained from $\bH$ by ribbon graph isomorphism, vertex joins, vertex cuts, and mutations; see~\cite{MO2021} for the definitions of these operations.

\subsubsection{Strong $\Delta$-matroids}\label{sec: strong delta-matroids}

A $\Delta$-matroid $(E,\cB)$ is \emph{strong} if it satisfies the \emph{strong basis exchange property}:
\begin{center}
    for any $B,B'\in \cB$ and any $x \in B\symdiff B'$, there is $y \in B\symdiff B'$ such that $B\symdiff\{x,y\}, B'\symdiff\{x,y\} \in \cB$.
\end{center}
It is well known that matroids and even $\Delta$-matroids satisfy the strong basis exchange property~\cite{Wenzel1993}, but general $\Delta$-matroids do not have this property.
Geelen and Murota~\cite{Murota2021} showed that strong $\Delta$-matroids are equivalent to even $\Delta$-matroids through the \emph{lift} of set systems, which we define now.

\begin{definition}\label{def: hat}
    For a finite set $E$, we define a map $\alpha_E$ assigning to a subset $I$ of $E$ a set 
    \[
        \alpha_E(I) :=
        \begin{cases}
            I & \text{if $|I|$ is even},\\
            I \cup \{\widehat{e}\} & \text{if $|I|$ is odd},\\
        \end{cases}
    \]
    where $\widehat{e}$ is a fixed element not in $E$.
    We identify $\widehat{e} = n+1$ when $E = [n]$.
    We often omit the subscript and write $\alpha_E$ as $\alpha$.
\end{definition}

\begin{definition}\label{def: lift}
    The \emph{lift} of a finite set system $S = (E,\cB)$ is $\widehat{S} := (E\cup\{\widehat{e}\}, \alpha(\cB))$ where $\alpha(\cB) := \{\alpha(I) : I\in \cB\}$.
\end{definition}

\begin{theorem}[{\cite[Thm.~3.1]{Murota2021}},{\cite[Remark~ 3.10]{CDFS2025}}]\label{thm: strong and even delta-matroids}
    Let $D=(E,\cB)$ be a set system.
Then $D$ is a strong $\Delta$-matroid if and only if $\widehat{D}$ is an even $\Delta$-matroid.
\end{theorem}

\begin{corollary}\label{cor: BD bijection} The map
  \begin{align*}
\{\text{strong $\Delta$-matroids on $E$}\}
&\to
\Big\{
\hspace{-1.5mm}
\begin{array}{c}
\text{even $\Delta$-matroids on } E\cup \{\widehat{e}\} \\[-2pt]
\text{with even-sized bases}
\end{array}
\hspace{-1.5mm}
\Big\} \\
D \hspace{2.0cm} &\mapsto \hspace{2.5cm} \widehat{D}
\end{align*}
is a bijection. 
\end{corollary}

\begin{proof}
By Definition~\ref{def: hat}, $\alpha_E$ is a bijection between subsets of $E$ and even-sized subsets of $E\cup \{\widehat{e}\}$. This implies that $S\mapsto \widehat{S}$ in Definition~\ref{def: lift} is a bijection between set systems $(E,\mathcal{B})$ and the set systems $(E\cup\{\widehat{e}\},\mathcal{B}')$ where all subsets in $\mathcal{B}'$ are even-sized. Then by Theorem~\ref{thm: strong and even delta-matroids}, the desired bijection $D\mapsto \widehat{D}$ is a restriction of the bijection $S\mapsto \widehat{S}$.
\end{proof}

Now we study the above bijection in the representable case. The following lemma is implicit in the proof of ~\cite[Cor.~6.12]{CCLV2025}.

\begin{lemma}\label{lem: determinants of type B and D}
    Let $\bfA$ be an $n$-by-$n$ skew-symmetric matrix over a field $K$ and $\bfv\in K^n$. Denote 
    \[
        \bfA_{\bfv} = 
        \begin{pmatrix}
            \bfA & \bfv \\
            -\bfv^T & 0 \\
        \end{pmatrix}.
    \]
    Then for each $I \subseteq [n]$,
    \[
        \det((\bfA + \bfv \bfv^T)[I]) = \det({\bfA_\bfv}[\alpha(I)]).
    \]
\end{lemma}
\begin{proof}
    Note that 
    \[
        \det(\bfA + \bfv \bfv^T) 
        = 
        \det
        \begin{pmatrix}
            \bfA + \bfv \bfv^T & \mathbf{0} \\
            -\bfv^T & 1 \\
        \end{pmatrix}
        =
        \det
        \begin{pmatrix}
            \bfA & \bfv \\
            -\bfv^T & 1 \\
        \end{pmatrix}
        =
        \det(\bfA) + \det(\bfA_{\bfv}),
    \]
    where the last equality follows from the cofactor expansion along the last row.
    By replacing $\bfA$ and $\bfv$ with $\bfA[I]$ and $\bfv[I]$, we obtain 
    \[
        \det((\bfA + \bfv \bfv^T)[I]) = \det(\bfA[I]) + \det(\bfA_{\bfv}[I\cup\{n+1\}]).
    \]
    Because every odd-sized skew-symmetric matrix is singular, we deduce the desired identity.
\end{proof}

\begin{proposition}\label{prop: skew-sym plus rank-one is strong}
        Let $\bfA$ be an $n$-by-$n$ skew-symmetric matrix over a field $K$ and $\bfv\in K^n$. Then the set system $D = ([n],\cB)$ is a strong $\Delta$-matroid, where
        \[
            \cB = 
            \{ I \subseteq [n] : (\bfA + \bfv \bfv^T) [I] \text{ is nonsingular} \}.
        \]
\end{proposition}
\begin{proof}
    By Lemma~\ref{lem: determinants of type B and D}, $\widehat{D}$ is represented by a skew-symmetric matrix $\bfA_{\bfv}$, and thus $\widehat{D}$ is an even $\Delta$-matroid.
    By Theorem~\ref{thm: strong and even delta-matroids}, $D$ is a strong $\Delta$-matroid.
\end{proof}

\begin{lemma}\label{lem: sym is skew-sym plus rank-one in char two}
    Every symmetric matrix $\bfM$ over a finite field of characteristic two can be uniquely written as $\bfA + \bfv \bfv^T$ for a skew-symmetric matrix $\bfA$ and a vector $\bfv$.    
\end{lemma}
\begin{proof}
    Note that every element in the field has a square root. Since the diagonal entries $\bfA$ are zero, the diagonal entries of $\bfv \bfv^T$ must agree with those of $\bfM$, which implies that $\bfv$ satisfies $\bfv(i)^2 = \bfM_{ii}$ for each $i$. Then $\bfA := \bfM - \bfv \bfv^T$ is the skew-symmetric matrix with the desired property. 
\end{proof}

For a symmetric matrix $\bfM = (a_{ij})_{1\le i,j \le n}$ over a finite field of characteristic two, we define the skew-symmetric matrix
\[\widehat{\bfM}:=\bfA_\bfv,\]
where $\bfA$ and $\bfv$ are as in Lemma~\ref{lem: sym is skew-sym plus rank-one in char two}. In terms of the entries, the formula is 
\[
    \widehat{\bfM}_{ij} = 
    \begin{cases}
        0                               & \text{if $i=j$}, \\
        a_{ij} + \sqrt{a_{ii}a_{jj}}    & \text{if $i,j\in[n]$ and $i\ne j$}, \\
        \sqrt{a_{ii}}                     & \text{if $i\in[n]$ and $j=n+1$} , \\
        \sqrt{a_{jj}}                     & \text{if $j\in[n]$ and $i=n+1$} . \\
    \end{cases}
\]

\begin{proposition}\label{prop: sym to skew-sym}
    $\widehat{D(\bfM)} = D(\widehat{\bfM})$ for
    a symmetric matrix $\bfM$ over a finite field of characteristic two.
\end{proposition}
\begin{proof}
    It is a direct consequnce of Lemma~\ref{lem: determinants of type B and D}.
\end{proof}

\begin{proposition}\label{prop: representability and strong}
    Every $\Delta$-matroid representable over a finite field of characteristic two is strong.
\end{proposition}
\begin{proof}
    As strongness is preserved under twisting, we may assume that the $\Delta$-matroid $D = D(\bfM)$ for some symmetric matrix $\bfM$.
    Then $D$ is strong by Lemma~\ref{lem: sym is skew-sym plus rank-one in char two} and Proposition~\ref{prop: skew-sym plus rank-one is strong}.
\end{proof}

\begin{remark}
The smallest non-strong $\Delta$-matroid $([3],\{\emptyset, 1, 2, 3, 123\})$ is representable over any field of characteristic not two by the $3$-by-$3$ symmetric matrix whose all diagonal entries are $1$ and all off-diagonal entries are $-1$.
\end{remark}

Because every ribbon-graphic $\Delta$-matroid is binary by Theorem~\ref{thm: ribbon graphic is binary}, we obtain the following corollary.

\begin{corollary}\label{cor: ribbon is strong}
    Every ribbon-graphic $\Delta$-matroid is strong.
    \qed
\end{corollary}

The main purpose of this paper is to understand a geometric operation that converts certain ribbon graphs $\bG$ to $\widehat{\bG}$ so that $\widehat{D(\bG)} = D(\widehat{\bG})$, and to characterize the ribbon graphs that allow such an operation.
We will see these definitions and results in \S\ref{sec: pseudo-orientable ribbon graphs}.

\begin{remark}    
    Let $\bfM = (a_{ij})_{1\le i,j\le n}$ be a symmetric matrix over a (possibly infinite) field of characteristic two, and let $\widetilde{\bfM}$ be the $(n+1)\times(n+1)$ skew-symmetric matrix defined as
    \[
        \widetilde{\bfM}_{ij}
        :=
        \begin{cases}
            0 & \text{if $i=j$}, \\
            a_{ij}^2 + a_{ii}a_{jj} & \text{if $i,j\in[n]$ and $i\ne j$}, \\
            a_{ii} & \text{if $i\in[n]$ and $j=n+1$} \\
            a_{jj} & \text{if $j\in[n]$ and $i=n+1$}.
        \end{cases}
    \]
    Then van Geemen and Marrani~\cite{vGM2019} showed  
    \[
        \det(\bfM[I]) = \pf(\widetilde{\bfM}[\alpha(I)])
    \]
    for any $I\subseteq [n]$.
    Their result together with Theorem~\ref{thm: strong and even delta-matroids} implies a stronger version of Proposition~\ref{prop: representability and strong}, that is, every $\Delta$-matroid representable over a field of characteristic two is strong.

    Over a finite field of characteristic two, $\widetilde{\bfM}$ is the Hadamard square of $\widehat{\bfM}$, i.e., $\widetilde{\bfM}_{ij} = \widehat{\bfM}_{ij}^2$ for all $i,j$, and for any skew-symmetric matrix $\bfC$, we have
    \[
        \pf(\bfC^{\circ 2}) = \sum_{\sigma \in \mathfrak{S}_n} \prod_{i} \bfC_{i\sigma(i)}^2 = \left( \sum_{\sigma \in \mathfrak{S}_n} \prod_{i} \bfC_{i\sigma(i)} \right)^2
        = \pf(\bfC)^2 = \det(\bfC).
    \]
    Hence, Lemma~\ref{lem: determinants of type B and D} implies van Geemen and Marrani's result for finite fields of characteristic two.
\end{remark}

\section{Pseudo-orientable ribbon graphs}\label{sec: pseudo-orientable ribbon graphs}

In \S\ref{sec: define pseudo-orientability}, we define \emph{pseudo-orientable ribbon graphs} $\bG$ and the \emph{adjustment} $\widehat{\bG}$. Moreover, we show that the pseudo-orientability is closed under taking minors (Proposition~\ref{prop: partial duality preserves pseudo-orientability}).
In \S\ref{sec: certificate}, we address a subtlety in the definition of pseudo-orientability, though readers may skip this subsection without loss of continuity. In \S\ref{sec: prove thm1}, we prove Theorem~\ref{thm-intro: pseudo-orientable and orientable}, which asserts that pseudo-orientable ribbon graphs are the only class of ribbon graphs that satisfy the property $\widehat{D(\bG)} = D(\widehat{\bG})$ up to the $2$-isomorphism of ribbon graphs introduced in~\cite{MO2021}. 
In \S\ref{sec: matrix--quasi-tree}, we prove a Matrix--Quasi-tree Theorem for pseudo-orientable ribbon graphs (Theorem~\ref{thm-intro: matrix--quasi-tree for pseudo-orientable}).

\subsection{Pseudo-orientable ribbon graphs and their adjustments}\label{sec: define pseudo-orientability}
We first define pseudo-orientability for bouquets.

\begin{definition}\label{def: pseudo-orientable bouquet}
    A bouquet is \emph{pseudo-orientable} if the boundary of the vertex admits two closed segments $S_1$ and $S_2$ such that 
    \begin{itemize}
        \item $S_1\cap S_2$ contains exactly two points,
        \item the two ends of each orientable loop lie in the interior of one of $S_1$ and $S_2$, and 
        \item the two ends of each non-orientable loop lie in the interior of $S_1$ and the interior of $S_2$, respectively.
    \end{itemize}
We call $(S_1,S_2)$ a \emph{certificate}.
    The \emph{adjustment} $\alpha(\bB,S_1,S_2)$ of $\bB$ at $(S_1,S_2)$ is the orientable bouquet obtained in the following way:
    \begin{enumerate} \item Cut the bouquet along $S_2$ and then reglue it with a half-twist. Note that all loops are orientable in the new bouquet.
        \item Add a new orientable loop, denoted by $\widehat{e}$, connecting the two points in $S_1\cap S_2$.
    \end{enumerate}
\end{definition}

See Figure~\ref{fig: pseudo-orientable bouquet} for an example. Also note that every bouquet with at most one non-orientable loop is pseudo-orientable.

We often denote the adjustment by $\widehat{\bB}$ if the certificate is clear from or immaterial to the context.

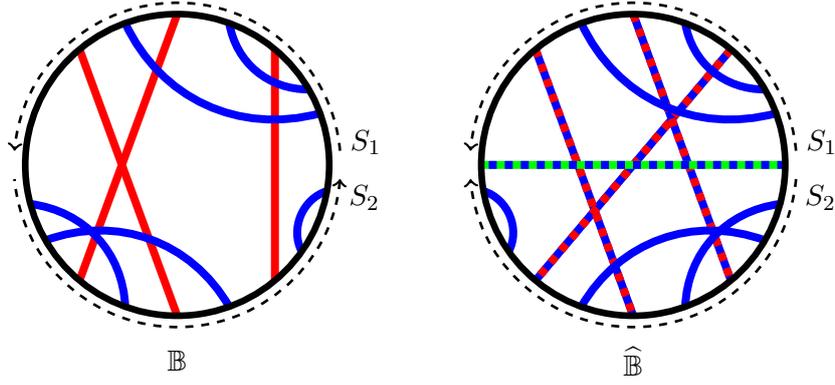
\begin{figure}[h!]
    \centering
    \begin{tikzpicture}
\begin{scope}
            \draw[line width=3pt, red] (50:2) -- (-50:2);
            \draw[line width=3pt, red] (90:2) -- (-130:2);
            \draw[line width=3pt, red] (130:2) -- (-90:2);
            
            \draw[line width=3pt, blue] (30:2) to[bend left=40] (70:2);
            \draw[line width=3pt, blue] (20:2) to[bend left=40] (110:2);

            \draw[line width=3pt, blue] (-165:2) to[bend left=40] (-110:2);
            \draw[line width=3pt, blue] (-150:2) to[bend left=40] (-70:2);
            \draw[line width=3pt, blue] (-10:2) to[bend right=50] (-35:2);

            \draw[line width=2.5pt] (0,0) circle (2);
            
            \draw[line width=1pt, dashed, ->] (5:2.15) arc (5:175:2.15);
            \draw[line width=1pt, dashed, <-] (-5:2.15) arc (-5:-175:2.15);
            \node at (7:2.5) {$S_1$};
            \node at (-10:2.5) {$S_2$};

            \node at (0,-2.6) {$\bB$};
        \end{scope} 

        \begin{scope}[xshift=6cm]
            \draw[line width=3pt, blue] (50:2) -- (-130:2);
            \draw[line width=3pt, blue] (90:2) -- (-50:2);
            \draw[line width=3pt, blue] (130:2) -- (-90:2);
            \draw[line width=3pt, dashed, red] (50:2) -- (-130:2);
            \draw[line width=3pt, dashed, red] (90:2) -- (-50:2);
            \draw[line width=3pt, dashed, red] (130:2) -- (-90:2);
            
            \draw[line width=3pt, blue] (30:2) to[bend left=40] (70:2);
            \draw[line width=3pt, blue] (20:2) to[bend left=40] (110:2);

            \draw[line width=3pt, blue] (-15:2) to[bend right=40] (-70:2);
            \draw[line width=3pt, blue] (-30:2) to[bend right=40] (-110:2);
            \draw[line width=3pt, blue] (-170:2) to[bend left=50] (-145:2);

            \draw[line width=3pt, green] (0:2) -- (180:2);
            \draw[line width=3pt, dashed, blue] (0:2) -- (180:2);

            \draw[line width=1pt, dashed, ->] (5:2.15) arc (5:175:2.15);
            \draw[line width=1pt, dashed, ->] (-5:2.15) arc (-5:-175:2.15);
            \node at (7:2.5) {$S_1$};
            \node at (-10:2.5) {$S_2$};

            \draw[line width=2.5pt] (0,0) circle (2);
            
            \node at (0,-2.6) {$\widehat{\bB}$};
        \end{scope} 
    \end{tikzpicture}
    \caption{The left figure is a pseudo-orientable bouquet $\bB$ whose orientable loops are colored in blue and non-orientable loops are colored in red. A certificate $(S_1,S_2)$ is indicated by dashed arrows. The right figure is the adjustment $\widehat{\bB}$ of $\bB$ at $(S_1,S_2)$, which is obtained by flipping the bottom segment $S_2$. The new loop $\widehat{e}$ is depicted as the blue-green dashed line. All loops in $\widehat{\bB}$ are orientable.
    }
    \label{fig: pseudo-orientable bouquet}
\end{figure}

\begin{remark}
    We say that two certificates $(S_1,S_2)$ and $(S_1',S_2')$ of a pseudo-orientable bouquet are \emph{equivalent} if the loop ends intersecting with $S_1$ are equal to those of $S_1'$ or those of $S_2'$.
    If two certificates are equivalent, their adjustments are identical.
\end{remark}

The inverse of adjusting pseudo-orientable bouquets can be described as a sequence of three basic operations---Petrial dual, partial dual, and edge deletion---applied to a ribbon graph at the same edge. To be precise, we have the following two results, which can be proved by the definition of $\alpha(\bB,S_1,S_2)$.

\begin{lemma}\label{lem: adjustment1}
Let $\bB$ be a pseudo-orientable bouquet with a certificate $(S_1,S_2)$. 
    Then $\alpha(\bB,S_1,S_2)$ is an orientable bouquet, and
\[
        \bB = ( \alpha(\bB,S_1,S_2)^{\tau({e})} )^{{e}} \setminus {e},
    \]
    where $e= \widehat{e}$ is the new loop added in the adjustment $\alpha(\bB,S_1,S_2)$.
\end{lemma}

\begin{lemma}\label{lem: adjustment2}
     Let $\bB$ be an orientable bouquet and $e$ be one of its loops. 
    Then $(\bB ^{\tau({e})} )^{{e}} \setminus {e}$ is a pseudo-orientable bouquet and
\[
        \alpha((\bB ^{\tau({e})} )^{{e}} \setminus {e},S_1,S_2) = \bB,
    \]
    where we let the certificate $(S_1,S_2)$ be such that $e$ connects the two points in $S_1\cap S_2$, and when we apply $\alpha$, the new edge is labeled by $e$. 
\end{lemma}

\begin{proposition}\label{prop: reverse of adjustment}  
Every orientable bouquet with at least one edge is an adjustment of a pseudo-orientable bouquet.        
\end{proposition}

For a general ribbon graph, pseudo-orientability is defined as follows. 
\begin{definition}\label{def: pseudo-orientable ribbon graph}
    A connected ribbon graph $\bG$ is \emph{pseudo-orientable} if it is a partial dual of a pseudo-orientable bouquet, i.e., $\bG^X$ is a pseudo-orientable bouquet for some quasi-tree $X$ of $\bG$.
    We say a ribbon graph is \emph{pseudo-orientable} if one of the components is pseudo-orientable and all the other components are orientable. 
\end{definition}

\begin{remark}\label{rem: two definitions of pseudo-orientable}
    We will show in Section~\ref{sec: certificate} that if a bouquet is pseudo-orientable in the sense of Def.~\ref{def: pseudo-orientable ribbon graph}, then it is still pseudo-orientable in the sense of Def.~\ref{def: pseudo-orientable bouquet}. Hence these two definitions are compatible. 
\end{remark}

We have a simple criterion for pseudo-orientability of ribbon graphs.

\begin{lemma}\label{lem: simple criterion for pseudo-orientability}
    A ribbon graph $\bG$ is pseudo-orientable if and only if $\bG = (\bH^{\tau(e)})^e \setminus e$ for some orientable ribbon graph $\bH$ and some edge $e$ of $\bH$.
\end{lemma}
\begin{proof}
    Without loss of generality, we may assume that $\bG$ is connected.

    Suppose that $\bG$ is pseudo-orientable, i.e., $\bG = \bB^X$ for some pseudo-orientable bouquet $\bB$.
    Then there are an adjustment $\widehat{\bB}$ of $\bB$ and an edge $e$ of $\widehat{\bB}$ such that $\bB = (\widehat{\bB}^{\tau(e)})^e \setminus e$. Let $\bH := \widehat{\bB}^X$ by Lemma~\ref{lem: adjustment1}. 
    Because $X\cap \{e\} = \emptyset$, we have
    \[
        (\bH^{\tau(e)})^e \setminus e = ((\widehat{\bB}^X)^{\tau(e)})^e \setminus e = ((\widehat{\bB}^{\tau(e)})^e \setminus e)^X = \bB^X = \bG.
    \]

    To prove the converse, suppose that $\bG = (\bH^{\tau(e)})^e \setminus e$ for some orientable ribbon graph $\bH$ and some edge $e$ of $\bH$. 
    If there is a quasi-tree $X$ of $\bH$ such that $e\notin X$ and $\bH^X$ is a bouquet, then $\bG^X = ((\bH^{\tau(e)})^e \setminus e)^X = ((\bH^X)^{\tau(e)})^e \setminus e$ is a pseudo-orientable bouquet by Lemma~\ref{lem: adjustment2}, and thus $\bG$ is a pseudo-orientable ribbon graph.
    Thus, we may assume that $e$ is a bridge of $\bH$.
    Then $e$ is also a bridge of $\bH^{\tau(e)}$, and $\bH^{\tau(e)}$ is orientable. Therefore, $\bG$ is also orientable, so it is a partial dual of an orientable bouquet. As every orientable bouquet is pseudo-orientable, we conclude that $\bG$ is a pseudo-orientable ribbon graph.
\end{proof}

As a consequence, we deduce the following:

\begin{proposition}\label{prop: partial duality preserves pseudo-orientability}
    The class of pseudo-orientable ribbon graphs is minor-closed.
\end{proposition}
\begin{proof}
    Let $\bG$ be a pseudo-orientable ribbon graph and let $X,Y$ be disjoint subsets of $E(\bG)$. We claim that $\bG^X \setminus Y$ is pseudo-orientable.

    By Lemma~\ref{lem: simple criterion for pseudo-orientability}, there is an orientable ribbon graph $\bH$ and an edge $e$ of $\bH$ such that $\bG = (\bH^{\tau(e)})^e \setminus e$. 
    We have
    \[
        \bG^X \setminus Y = ((\bH^{\tau(e)})^e \setminus e)^X \setminus Y = ((\bH^X\setminus Y)^{\tau(e)})^e \setminus e.
    \]
    Since $\bH^X\setminus Y$ is orientable, we conclude that $\bG^X \setminus Y$ is pseudo-orientable by Lemma~\ref{lem: simple criterion for pseudo-orientability}.
\end{proof}

\begin{definition}
    Let $\bG$ be a connected pseudo-orientable ribbon graph and $X$ be a quasi-tree of $\bG$ such that $\bG^X$ is a pseudo-orientable bouquet.
    We define the \emph{adjustment} of $\bG$ at $X$ as 
    \[
        \widehat{\bG} :=
        \begin{cases}
            ( \widehat{\bG^X} )^X      & \text{if $|X|$ is even}, \\
            ( \widehat{\bG^X} )^{X\cup\{\widehat{e}\}}
                                        & \text{otherwise}.
        \end{cases}
    \]
    For a general pseudo-orientable ribbon graph, its \emph{adjustment} is defined by taking an adjustment of the pseudo-orientable component and leaving the other components unchanged.
\end{definition}

The definition of $\widehat{\bG}$ depends on the choices of the quasi-tree $X$ and the certificate for $\bG^X$.
However, since we only care about the associated $\Delta$-matroid $D(\widehat{\bG})$, these choices do not matter; see Section~\ref{sec: prove thm1}.

It is clear that an adjustment $\widehat{\bG}$ is an orientable ribbon graph. We remark that not every orientable ribbon graph is an adjustment of a pseudo-orientable ribbon graph, since the parity of a quasi-tree of an adjustment is even by definition.
We finally note the following counterpart of Proposition~\ref{prop: reverse of adjustment}.

\begin{proposition}
    Every orientable ribbon graph with at least one edge and only even-sized quasi-trees is an adjustment of a pseudo-orientable ribbon graph.
\end{proposition}
\begin{proof}
    Let $\bH$ be an orientable ribbon graph satisfying the assumption. We may assume that $\bH$ is connected. 
    If $\bH$ is a tree, then it is readily seen that $\bH$ is an adjustment of a tree obtained by contracting a leaf edge of $\bH$. Thus, we may assume that $\bH$ is not a tree.
    Then $\bH$ has a quasi-tree $X$ that is a proper subset of $E(\bH)$. Let $e$ be an edge in $E(\bH)\setminus X$. 
    One can easily check that $\bG := (\bH^{\tau(e)})^e \setminus e$ is a pseudo-orientable ribbon graph which has $\bH$ as its adjustment.
\end{proof}

\subsection{Pseudo-orientability of bouquets}\label{sec: certificate}
In this subsection, we show that the two definitions of pseudo-orientability for bouquets (Def.~\ref{def: pseudo-orientable bouquet} and Def.~\ref{def: pseudo-orientable ribbon graph}) are equivalent (cf. Remark~\ref{rem: two definitions of pseudo-orientable}). While this equivalence is not strictly necessary for the remainder of the paper---provided we specify which definition is used in each instance---establishing it settles a natural question and unifies the two notions.

It suffices to prove the following lemma.

\begin{lemma}\label{lem: partial duality preserves pseudo-orientability}
    Let $\bB$ be a bouquet with a certificate of pseudo-orientability.
    Then any bouquet that is a partial dual of $\bB$ has a certificate.
\end{lemma}
\begin{proof}
    Let $\bB$ be a pseudo-orientable bouquet with a certificate $(S_1,S_2)$.
    By Proposition~\ref{prop: elementary partial duals}, it suffices to show that the pseudo-orientability is preserved under elementary partial duality.

    We first show that the partial dual $\bB^{\{e,f\}}$ is pseudo-orientable whenever $e$ and $f$ are an interlacing pair of orientable loops in $\bB$.
    By symmetry, we may assume that the ends of $e$ and $f$ are in $S_2$.
    By cutting the boundary of the vertex at the four ends of $e$ and $f$, we obtain four segments $T_1,T_2,T_3,T_4$ that partition the boundary. We label them as in Figure~\ref{fig: partial duality preserves pseudo-orientability - orientable} so that $T_2, T_3, T_4$ are contained in $S_2$.
    Then $\bB^{\{e,f\}}$ can be obtained by swapping the segments $T_2$ and $T_4$.
    Let $S_1' := S_1$ and $S_2'$ be the segments of the boundary of the vertex of $\bB^{\{e,f\}}$ obtained from $S_2$ by swapping $T_2$ and $T_4$.
    
    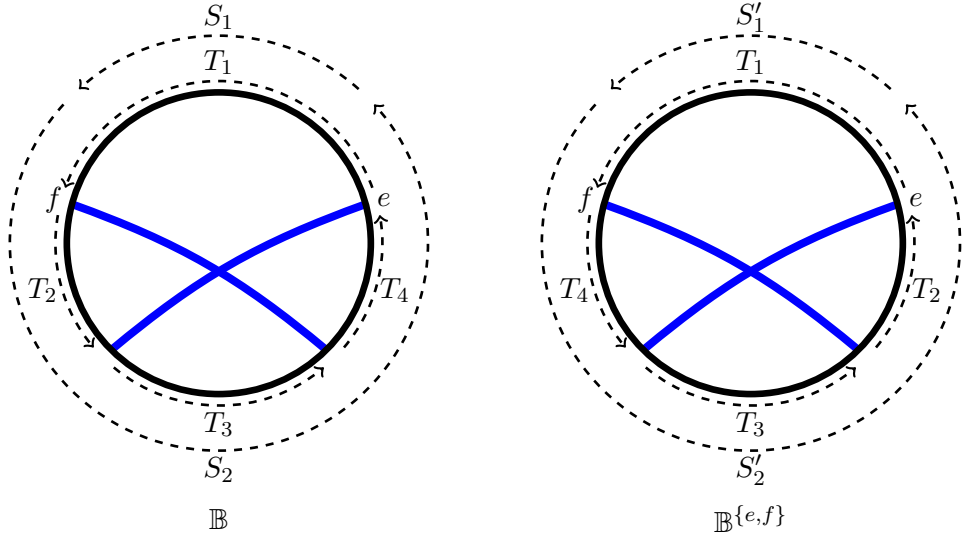
\begin{figure}[h!]
        \centering
        \begin{tikzpicture}
            \begin{scope}
                \draw[line width=1pt, dashed, ->] (50-2:2.75) arc (50-2:130+2:2.75);
                \draw[line width=1pt, dashed, ->] (140-2:2.75) arc (140-2:400+2:2.75);

                \node at (90:3) {$S_1$};
                \node at (270:3) {$S_2$};

                \draw[line width=3pt, blue] (15:2) to[bend right=10] (-135:2);
                \draw[line width=3pt, blue] (165:2) to[bend left=10] (-45:2);

                \draw[line width=1pt, dashed, ->] (20:2.15) arc (20:160:2.15);
                \draw[line width=1pt, dashed, ->] (170:2.15) arc (170:220:2.15);
                \draw[line width=1pt, dashed, ->] (230:2.15) arc (230:310:2.15);
                \draw[line width=1pt, dashed, ->] (320:2.15) arc (320:370:2.15);

                \node at (90:2.40) {$T_1$};
                \node at (195:2.40) {$T_2$};
                \node at (-90:2.40) {$T_3$};
                \node at (-15:2.40) {$T_4$};

                \node at (15:2.25) {$e$};
                \node at (165:2.25) {$f$};

                \draw[line width=2.5pt] (0,0) circle (2);

                \node at (0,-3.65) {$\bB$};
            \end{scope} 
            \begin{scope}[xshift=7cm]
                \draw[line width=1pt, dashed, ->] (50-2:2.75) arc (50-2:130+2:2.75);
                \draw[line width=1pt, dashed, ->] (140-2:2.75) arc (140-2:400+2:2.75);

                \node at (90:3) {$S_1'$};
                \node at (270:3) {$S_2'$};

                \draw[line width=3pt, blue] (15:2) to[bend right=10] (-135:2);
                \draw[line width=3pt, blue] (165:2) to[bend left=10] (-45:2);

                \draw[line width=1pt, dashed, ->] (20:2.15) arc (20:160:2.15);
                \draw[line width=1pt, dashed, ->] (170:2.15) arc (170:220:2.15);
                \draw[line width=1pt, dashed, ->] (230:2.15) arc (230:310:2.15);
                \draw[line width=1pt, dashed, ->] (320:2.15) arc (320:370:2.15);

                \node at (90:2.40) {$T_1$};
                \node at (195:2.40) {$T_4$};
                \node at (-90:2.40) {$T_3$};
                \node at (-15:2.40) {$T_2$};

                \node at (15:2.25) {$e$};
                \node at (165:2.25) {$f$};

                \draw[line width=2.5pt] (0,0) circle (2);

                \node at (0,-3.65) {$\bB^{\{e,f\}}$};
            \end{scope}
        \end{tikzpicture}
        \caption{
            The first case in the proof of Lemma~\ref{lem: partial duality preserves pseudo-orientability}.
        }   
        \label{fig: partial duality preserves pseudo-orientability - orientable}     
    \end{figure}

    \begin{claim}\label{claim: partial duality preserves pseudo-orientability - orientable}
        The pair $(S_1',S_2')$ certifies that $\bB^{\{e,f\}}$ is pseudo-orientable.
    \end{claim}
    \begin{proof}
        Let $g$ be a loop in $\bB$.
        Note that the orientability of $g$ is preserved under the elementary partial dual with respect to $e,f$.

        Suppose that $g$ is an orientable loop in $\bB$.
        If the ends of $g$ are in $S_1$, then the corresponding edge $g$ in $\bB^{\{e,f\}}$ has its ends in $S_1'$.
        So, we may assume that the ends of $g$ are in $S_2$.
        In this case, it is also easy to see that the ends of $g$ in $\bB^{\{e,f\}}$ are in $S_2'$.
        
        Suppose that $g$ is a non-orientable loop in $\bB$.
        Then, for $\bB$, one end of $g$ is in $S_1$ and the other end is in $S_2$.
        Thus, for $\bB^{\{e,f\}}$, one end of $g$ is in $S_1'$ and the other end is in~$S_2'$.
    \end{proof}

    It remains to show that the partial dual $\bB^{e}$ is pseudo-orientable whenever $e$ is a non-orientable loop in $\bB$.
    By cutting the boundary of the vertex at the ends of $e$ and the ends of $S_1$, we obtain the partition $(T_1,T_2,T_3,T_4)$ of the boundary as depicted in Figure~\ref{fig: partial duality preserves pseudo-orientability - non-orientable}.
    Then $\bB^{\{e\}}$ can be obtained by flipping the segment $T_1\cup T_4$.
    Let $S_1'$ be the segment in $\bB^{\{e\}}$ defined as the union of $T_2$ and $T_4$, and let $S_2'$ be the segment in $\bB^{\{e\}}$ defined as the union of $T_1$ and $T_3$.

    \begin{figure}[h!]
        \centering
        \begin{tikzpicture}
            \begin{scope}
                \draw[line width=1pt, dashed, ->] (0+3:2.75) arc (0+3:180-3:2.75);
                \draw[line width=1pt, dashed, ->] (180+3:2.75) arc (180+3:360-3:2.75);

                \node at (90:3) {$S_1$};
                \node at (270:3) {$S_2$};

                \draw[line width=3pt, red] (90:2) -- (270:2);

                \draw[line width=1pt, dashed, ->] (5:2.15) arc (5:85:2.15);
                \draw[line width=1pt, dashed, ->] (95:2.15) arc (95:175:2.15);
                \draw[line width=1pt, dashed, ->] (185:2.15) arc (185:265:2.15);
                \draw[line width=1pt, dashed, ->] (275:2.15) arc (275:355:2.15);

                \node at (45:2.45) {$T_1$};
                \node at (135:2.45) {$T_2$};
                \node at (225:2.45) {$T_3$};
                \node at (315:2.45) {$T_4$};

                \node at (90:2.25) {$e$};

                \draw[line width=2.5pt] (0,0) circle (2);

                \node at (0,-3.65) {$\bB$};
            \end{scope} 
            \begin{scope}[xshift=7cm]
                \draw[line width=1pt, dashed, ->] (0+3:2.75) arc (0+3:180-3:2.75);
                \draw[line width=1pt, dashed, ->] (180+3:2.75) arc (180+3:360-3:2.75);

                \node at (90:3) {$S_1'$};
                \node at (270:3) {$S_2'$};

                \draw[line width=3pt, red] (90:2) -- (270:2);

                \draw[line width=1pt, dashed, <-] (5:2.15) arc (5:85:2.15);
                \draw[line width=1pt, dashed, ->] (95:2.15) arc (95:175:2.15);
                \draw[line width=1pt, dashed, ->] (185:2.15) arc (185:265:2.15);
                \draw[line width=1pt, dashed, <-] (275:2.15) arc (275:355:2.15);

                \node at (45:2.45) {$T_4$};
                \node at (135:2.45) {$T_2$};
                \node at (225:2.45) {$T_3$};
                \node at (315:2.45) {$T_1$};

                \node at (90:2.25) {$e$};

                \draw[line width=2.5pt] (0,0) circle (2);

                \node at (0,-3.65) {$\bB^{\{e\}}$};
            \end{scope}
        \end{tikzpicture}
        \caption{
            The second case in the proof of Lemma~\ref{lem: partial duality preserves pseudo-orientability}.
        }
        \label{fig: partial duality preserves pseudo-orientability - non-orientable}     
    \end{figure}
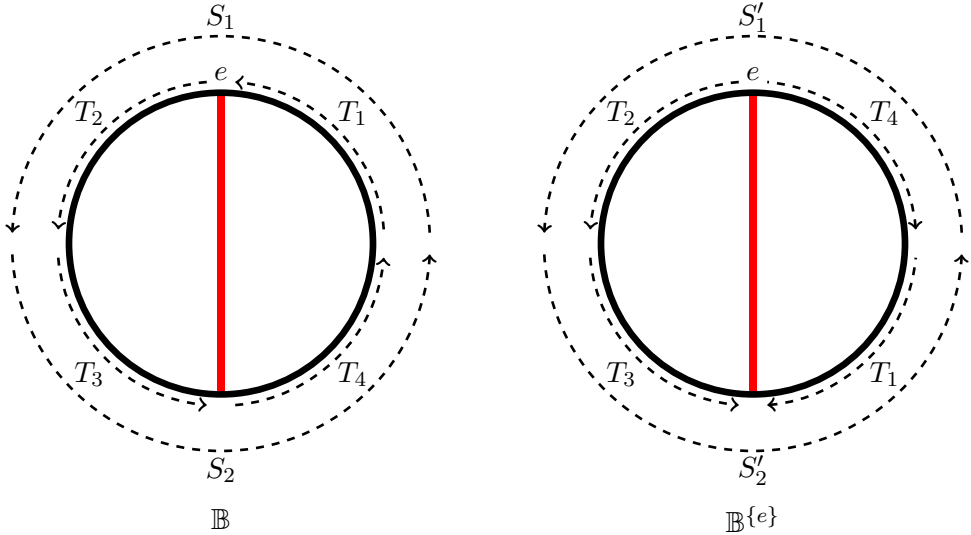

    \begin{claim}\label{claim: partial duality preserves pseudo-orientability - non-orientable}
        The pair $(S_1',S_2')$ certifies that $\bB^{\{e\}}$ is pseudo-orientable.
    \end{claim}
    \begin{proof}
        Let $g$ be an orientable loop in $\bB$.
        Then, the ends of $g$ are in the same part of the bipartition $(S_1,S_2)$.
        By symmetry, we may assume that the ends of $g$ are in $S_1$.
        If the ends of $g$ are in $T_1$, then the corresponding edge $g$ in $\bB^{\{e\}}$ is orientable and has its ends in $S_2'$.
        If the ends of $g$ are in $T_2$, then $g$ in $\bB^{\{e\}}$ is orientable and has its ends in $S_1'$. The last case is that one end of $g$ is in $T_1$ and the other end is in $T_2$.
        Then, $g$ is non-orientable in $\bB^{\{e\}}$ and has one end in $S_1'$ and the other end in $S_2'$.

        Next, let $g$ be a non-orientable loop in $\bB$.
        Then, one end of $g$ is in $S_1$ and the other end is in $S_2$.
        If the ends of $g$ in $\bB$ are in $T_1\cup T_4$ (so, one end is in $T_1$ and the other end is in $T_4$), then $g$ in $\bB^{\{e\}}$ is non-orientable and has one end in $S_1'$ and the other end in $S_2'$.
        In the case that the ends of $g$ are in $T_2 \cup T_3$, the same conclusion holds.
        The last case is that the ends of $g$ are in either $T_1\cup T_3$ or $T_2\cup T_4$.
        Without loss of generality, we may assume that they are in $T_1\cup T_3$ (so, one end is in $T_1$ and the other end is in $T_3$).
        Then $g$ in $\bB^{\{e\}}$ is orientable, and both ends are in $S_2'$.
    \end{proof}

    By Claims~\ref{claim: partial duality preserves pseudo-orientability - orientable} and~\ref{claim: partial duality preserves pseudo-orientability - non-orientable}, the elementary partial duality preserves pseudo-orientability.
\end{proof}

We remark that one can determine whether a given ribbon graph is pseudo-orientable in polynomial time and, moreover, find a certificate in the affirmative case.
Let $\bG$ be an input ribbon graph. We may assume that $\bG$ is connected. 
We first find a spanning quasi-tree $Q$ of $\bG$. It can be done efficiently, as every spanning tree of the underlying graph is a quasi-tree.
Next, we check whether there is a certificate $(S_1,S_2)$ for the bouquet $\bG^Q$.
Denote $n:=|E(\bG)|$. Then the boundary of $\bG^Q$ is partitioned into $2n$ segments by the ends of edges.
Thus, there are $\binom{2n}{2}+2n$ candidates for a certificate $(S_1, S_2)$ up to equivalence. One can check in polynomial time whether each candidate is a certificate.
If there is no certificate, then $\bG$ is not pseudo-orientable by Lemma~\ref{lem: partial duality preserves pseudo-orientability}.
Otherwise, we obtain a certificate and $\bG$ is pseudo-orientable.

\subsection{Delta-matroids associated with pseudo-orientable ribbon graphs}\label{sec: prove thm1}

We prove Theorem~\ref{thm-intro: pseudo-orientable and orientable}. 
We first show that lifts of set systems (Def.~\ref{def: lift}) and adjustments of ribbon graphs (Def.~\ref{def: pseudo-orientable ribbon graph}) are compatible
in the following sense. 
\begin{proposition}\label{prop: adjustment of ribbon graphs}
    If $\bG$ is a pseudo-orientable ribbon graph, then $\widehat{D(\bG)} = D(\widehat{\bG})$.
\end{proposition}

\begin{lemma}\label{lem: adjustment of bouquets}
    If $\bB$ is a pseudo-orientable bouquet, then $\bfM_2(\widehat{\bB}) = \widehat{\bfM_2(\bB)}$ and $D(\widehat{\bB})=\widehat{D(\bB)} $.
\end{lemma}
\begin{proof}
    The binary interlacing matrix $\bfM_2(\bB)$ is a binary representation of $D(\bB)$, which is defined as follows (Prop.~\ref{prop: binary interlacing matrix}):
    \[
        \bfM_2(\bB)_{ef} = 
        \begin{cases}
            1   &   \text{if $e=f$ and the loop is non-orientable}, \\
            0   &   \text{if $e=f$ and the loop is orientable}, \\
            1   &   \text{if $e\ne f$ and the two loops interlace}, \\
            0   &   \text{otherwise}.
        \end{cases}
    \]
    We denote the new edge in $\widehat{\bB}$ by $\widehat{e}$.
    Recall that $\widehat{\bB}$ is orientable.
    Then by the same construction, we have a binary representation $\bfM_2(\widehat{\bB})$ of $D(\widehat{\bB})$:
    \[
        \bfM_2(\widehat{\bB})_{ef} =
        \begin{cases}
            0   &   \text{if $e=f$}, \\
            1   &   \text{if $e\ne f$ and the two loops interlace}, \\
            0   &   \text{otherwise}.
        \end{cases}
    \]

    To show $\bfM_2(\widehat{\bB}) = \widehat{\bfM_2(\bB)}$, it suffices to verify equality of their off-diagonal entries. 

    Suppose $\widehat{e}\ne f$.
    By definition, the following four are equivalent:
    \begin{itemize}
        \item $\bfM_2(\widehat{\bB})_{\widehat{e}f} = 1$;
        \item $\widehat{e}$ and $f$ interlace in $\widehat{\bB}$;
        \item $f$ is non-orientable in $\bB$; and 
        \item $\bfM_2(\bB)_{ff} = 1$.
    \end{itemize}
    This implies $\bfM_2(\widehat{\bB})_{\widehat{e}f} = \bfM_2(\bB)_{ff}$. By symmetry, $\bfM_2(\widehat{\bB})_{f\widehat{e}} = \bfM_2(\bB)_{ff}$.

    Next, suppose that $e\ne f$ and neither $e$ nor $f$ is $\widehat{e}$.
    Then the following are equivalent:
    \begin{itemize}
        \item $\bfM_2(\widehat{\bB})_{ef} = 1$;
        \item $e$ and $f$ interlace in $\widehat{B}$;
        \item at least one of $e$ and $f$ is orientable and $e,f$ interlace, or both $e$ and $f$ are non-orientable and $e,f$ do not interlace; and 
        \item $\det(\bfM_2(\bB)[\{e,f\}]) = \bfM_2(\bB)_{ef} + \bfM_2(\bB)_{ee} \bfM_2(\bB)_{ff} = 1$.
    \end{itemize}
    Hence $\bfM_2(\widehat{\bB})_{ef} = \bfM_2(\bB)_{ef} + \bfM_2(\bB)_{ee} \bfM_2(\bB)_{ff}$.

    Therefore, $\bfM_2(\widehat{\bB}) = \widehat{\bfM_2(\bB)}$, which implies that 
    \[
        D(\widehat{\bB}) = D(\bfM_2(\widehat{\bB})) = D(\widehat{\bfM_2(\bB)}) = \widehat{D(\bfM_2(\bB))} = \widehat{D(\bB)},
    \]
    where the second last equality follows from Proposition~\ref{prop: sym to skew-sym}. 
\end{proof}

\begin{proof}[\bf Proof of Proposition~\ref{prop: adjustment of ribbon graphs}]
    Without loss of generality, we may assume that $\bG$ is connected. Let $X$ be a quasi-tree of $\bG$. Suppose that $|X|$ is even.
Then $\alpha(Q\symdiff X) = \alpha(Q)\symdiff X$ for each $Q \subseteq E(\bG)$.
    Moreover, the following are equivalent:
    \begin{itemize}
        \item $Q$ is a quasi-tree of $\bG$;
        \item $Q\symdiff X$ is a quasi-tree of $\bG^X$;
        \item $\alpha(Q\symdiff X)$ is a quasi-tree of $\widehat{\bG^X}$ (by Lemma~\ref{lem: adjustment of bouquets});
        \item $\alpha(Q) = \alpha(Q\symdiff X) \symdiff X$ is a quasi-tree of $\widehat{\bG} = (\widehat{\bG^X})^X.$
    \end{itemize}
    Thus, $\widehat{D(\bG)} = D(\widehat{\bG})$.

    Suppose that $|X|$ is odd. In this case, $\alpha(Q\symdiff X) = \alpha(Q) \symdiff X \symdiff \{\widehat{e}\}$. So, the above equivalences still hold except for the last one, which can be replaced with:
    \begin{itemize}
        \item $\alpha(Q) = \alpha(Q\symdiff X) \symdiff (X \cup \{\widehat{e}\})$ is a quasi-tree of $\widehat{\bG} = (\widehat{\bG^X})^{X \cup \{\widehat{e}\}}.$
        \qedhere
    \end{itemize}
\end{proof}

Now we are ready to prove our first main theorem. 

\begin{proof}[\bf Proof of Theorem~\ref{thm-intro: pseudo-orientable and orientable}]

    If $\bG$ is pseudo-orientable up to the $2$-isomorphism of ribbon graphs, then by Theorem~\ref{thm: 2-isomorphism} and Proposition~\ref{prop: adjustment of ribbon graphs}, the even $\Delta$-matroid $\widehat{D(\bG)}$ is ribbon-graphic. 
    
    Conversely, assume that $\bG$ has the property that $\widehat{D(\bG)}=D(\bH)$ for some orientable ribbon graph $\bH$. Without loss of generality, we may assume that $\bH$ is a disjoint union of bouquets. By Proposition~\ref{prop: reverse of adjustment}, there exists a pseudo-orientable ribbon graph $\bG'$ such that $\bH=\widehat{\bG'}$, and hence $\widehat{D(\bG)}=\widehat{D(\bG')}$. By Corollary~\ref{cor: BD bijection}, we obtain $D(\bG)=D(\bG')$. Then by Theorem~\ref{thm: 2-isomorphism}, $\bG$ is $2$-isomorphic to the pseudo-orientable ribbon graph $\bG'$. 
\end{proof}

We finally remark that a pseudo-orientable ribbon graph and a non-pseudo-orientable ribbon graph may induce the same $\Delta$-matroid.

\begin{example}
    There is a pair of pseudo-orientable and non-pseudo-orientable bouquets that induce the same $\Delta$-matroid.
    See Figure~\ref{fig: pseudo-orientable and non-pseudo-orientable bouquets}.
\end{example}

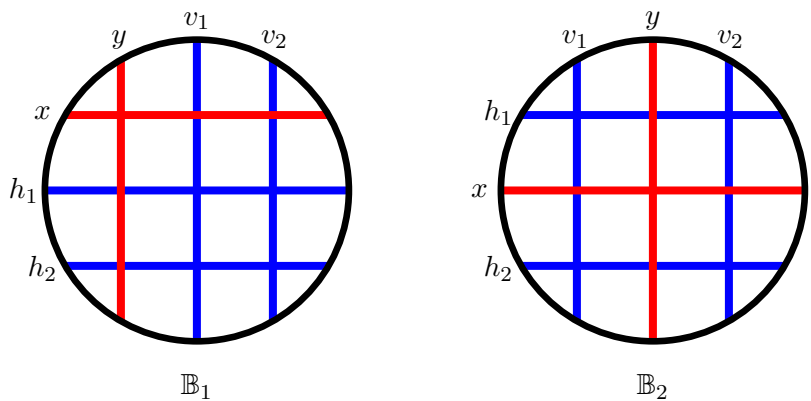
\begin{figure}[h!]
    \centering
    \begin{tikzpicture}
        \begin{scope}
\draw[line width=3pt, blue] (-30:2) -- (-150:2);
            \draw[line width=3pt, blue] (0:2) -- (180:2);

\draw[line width=3pt, blue] (60:2) -- (-60:2);
            \draw[line width=3pt, blue] (90:2) -- (-90:2);

\draw[line width=3pt, red] (30:2) -- (150:2);
            \draw[line width=3pt, red] (120:2) -- (-120:2);

            \node at (117:2.25) {$y$};
            \node at (90:2.25) {$v_1$};
            \node at (63:2.25) {$v_2$};

            \node at (153:2.28) {$x$};
            \node at (180:2.28) {$h_1$};
            \node at (207:2.28) {$h_2$};

            \draw[line width=2.5pt] (0,0) circle (2);

            \node at (0,-2.6) {$\bB_1$};
        \end{scope}
        \begin{scope}[xshift=6cm]
\draw[line width=3pt, blue] (30:2) -- (150:2);
            \draw[line width=3pt, blue] (-30:2) -- (-150:2);

\draw[line width=3pt, blue] (60:2) -- (-60:2);
            \draw[line width=3pt, blue] (120:2) -- (-120:2);

\draw[line width=3pt, red] (0:2) -- (180:2);
            \draw[line width=3pt, red] (90:2) -- (-90:2);

            \node at (117:2.25) {$v_1$};
            \node at (90:2.25) {$y$};
            \node at (63:2.25) {$v_2$};

            \node at (153:2.28) {$h_1$};
            \node at (180:2.28) {$x$};
            \node at (207:2.28) {$h_2$};

            \draw[line width=2.5pt] (0,0) circle (2);

            \node at (0,-2.6) {$\bB_2$};
        \end{scope}
    \end{tikzpicture}
    \caption{An pseudo-orientable bouquet $\bB_1$ (left) and a non-pseudo-orientable bouquet $\bB_2$ (right) with the edge set $\{x,y,v_1,v_2,h_1,h_2\}$.
    In both cases, the orientable loops $v_1,v_2,h_1,h_2$ are colored blue, and the non-orientable loops $x,y$ are colored red.
    Because the binary interlacing matrices are the same, $D(\bB_1) = D(\bB_2)$.
    }
    \label{fig: pseudo-orientable and non-pseudo-orientable bouquets}
\end{figure}

\subsection{Matrix--Quasi-tree Theorem for pseudo-orientable ribbon graphs}\label{sec: matrix--quasi-tree}

We prove the Matrix--Quasi-tree Theorem for pseudo-orientable ribbon graphs (Thm.~\ref{thm-intro: matrix--quasi-tree for pseudo-orientable}) by reducing it to the orientable case (Thm.~\ref{thm: matrix--quasi-tree theorem for orientable ribbon graphs}). 
\begin{definition} \label{def: real interlacing matrix}
    Let $\bB$ be an orientable bouquet.
    Assign an arbitrary orientation to each edge of $\bB$, and denote its tail (resp. head) by $e^-$ (resp. $e^+$).
    Choose one of the two orientations of the boundary of the unique vertex in $\bB$. Then we obtain a cyclic ordering of $2|E|$ elements $e^+,e^-$ with $e\in E$.

    The \emph{\textup{(}real\textup{)} interlacing matrix} $\bfM_\pm(\bB)$ is an $E(\bB)$-by-$E(\bB)$ skew-symmetric matrix with real entries, defined by
    \[
        \bfM_\pm(\bB)_{ef} := 
        \begin{cases}
            1   & \text{if the cyclic ordering is $\dots e^+ \dots f^+ \dots e^- \dots f^- \dots$}, \\
            -1  & \text{if the cyclic ordering is $\dots e^+ \dots f^- \dots e^- \dots f^+ \dots$}, \\
            0   & \text{otherwise, i.e., $e,f$ do not interlace or $e=f$}.
        \end{cases}
    \]
    
    Here we use the notation $\bfM_\pm$ since we view the matrix as a representation of the even $\Delta$-matroid $D(\bB)$ over the regular partial field $\bF_1^\pm$, while the binary interlacing matrix $\bfM_2(\bB)$ defined in Proposition~\ref{prop: binary interlacing matrix} is a representation of $D(\bB)$ over the finite field $\bF_2$.

\end{definition}

\begin{theorem}[\cite{MMN2025b}] \label{thm: matrix--quasi-tree theorem for orientable ribbon graphs}
    For an orientable bouquet $\bB$, 
    \[
        \det(\bfM_\pm(\bB)[X]) =
        \begin{cases}
            1   &   \text{$X$ is a quasi-tree of $\bB$}, \\
            0   &   \text{otherwise}.\\
        \end{cases}
    \]
    In particular, 
    \[
        \det(\bfI+\bfM_\pm(\bB)) = \text{the number of quasi-trees of $\bB$}.
    \]

\end{theorem}

\begin{remark}
    A different choice of an orientation of the boundary of $\bB$ in Def.~\ref{def: real interlacing matrix} only changes the matrix $\bfM_\pm(\bB)$ by scaling by $-1$.
    Similarly, a different choice of orientations for some edges changes the matrix by multiplying the corresponding rows and columns by $-1$.
    Thus, the determinant and, moreover, the Smith normal form of $\bfI + \bfM_\pm(\bB)$ are independent of such choices.
\end{remark}

For pseudo-orientable bouquets, we need to adjust the interlacing matrix by taking into account the non-orientable loops. The following definition is formulated so that the claim holds in the proof of Theorem~\ref{thm: matrix--quasi-tree theorem for pseudo-orientable ribbon graphs}. 

\begin{definition}\label{def: adjusted interlacing matrix}
    Let $\bB$ be a pseudo-orientable bouquet with a certificate $(S_1,S_2)$.
    Assign an arbitrary orientation to each orientable loop of $\bB$, and orient each non-orientable loop of $\bB$ so that its head is in $S_1$ and its tail is in $S_2$.
    For each edge $e$, denote its tail and head by $e^-$ and $e^+$, respectively.
    Choose one of the two orientations of the boundary of the unique vertex in $\bB$. This choice induces a cyclic ordering of $2|E|$ elements $e^+,e^-$ with $e\in E$.

    We define the \emph{adjusted interlacing matrix} $\bfM(\bB,S_1,S_2)$ as follows. Let $e,f\in E(\bB)$. If at least one of them is orientable, then
    \[
        \bfM(\bB,S_1,S_2)_{ef} :=
         \begin{cases}
            1   & \text{if the cyclic ordering is $\dots e^+ \dots f^+ \dots e^- \dots f^- \dots$}, \\
            -1  & \text{if the cyclic ordering is $\dots e^+ \dots f^- \dots e^- \dots f^+ \dots$}, \\
            0   & \text{otherwise, i.e., $e,f$ do not interlace or $e=f$}.
        \end{cases}
    \]
    If neither $e$ nor $f$ is orientable, we write $e<f$ if $e^+$ appears before $f^+$ in the linear order induced by the orientation of the segment $S_1$.
    Then, we define
    \[
        \bfM(\bB,S_1,S_2)_{ef} := 
        \begin{cases}
            1   & \text{if $e=f$ or they interlace}, \\
            2  & \text{if $e<f$ and they do not interlace}, \\
            0   & \text{if $e>f$ and they do not interlace}. \\
        \end{cases}
    \]
    For instance, see the left illustration in Figure~\ref{fig: different SNF} and the corresponding matrix in Example~\ref{ex: different SNF}.
\end{definition}

\begin{remark}
    The choices of the orientation in the definition above are not substantial. If we consider arbitrary orientations for non-orientable loops in Definition~\ref{def: adjusted interlacing matrix}, then the interlacing adjusted matrix $\bfM(\bB,S_1,S_2)$ can be defined by replacing the last case, non-orientable $e$ and $f$, as follows:
    We write $e<f$ if the end of $e$ in $S_1$ appears before the end of $f$ in $S_1$ in the linear ordering induced by the orientation of the segment $S_1$.
    We say that the orientations of $e,f$ \emph{agree} if $e^+$ and $f^+$ are in the same segment $S_1$ or $S_2$.
    \[
        \bfM(\bB,S_1,S_2)_{ef} = 
        \begin{cases}
            1   & \text{if $e=f$ or they interlace}, \\
            2  & \text{if $e<f$, they do not interlace, and the orientations of $e,f$ agree}, \\
            -2  & \text{if $e<f$, they do not interlace, and the orientations of $e,f$ do not agree}, \\
            0   & \text{if $e>f$ and they do not interlace}. \\
        \end{cases}
    \]
\end{remark}

Now we prove 
Theorem~\ref{thm-intro: matrix--quasi-tree for pseudo-orientable} for pseudo-orientable bouquets.

\begin{theorem}[Matrix--Quasi-tree Theorem for pseudo-orientable bouquets]
    \label{thm: matrix--quasi-tree theorem for pseudo-orientable ribbon graphs}
    For any pseudo-orientable ribbon bouquet $\bB$ with a certificate $(S_1,S_2)$,
    \[
        \det(\bfM(\bB,S_1,S_2)[X]) =
        \begin{cases}
            1   &   \text{$X$ is a quasi-tree of $\bB$}, \\
            0   &   \text{otherwise}.\\
        \end{cases}
    \]
    In particular,
    \[
        \det(\bfI+\bfM(\bB,S_1,S_2)) = \text{the number of quasi-trees of $\bB$}.
    \]
\end{theorem}
\begin{proof}
    Let $\widehat{\bB} := \bfM(\bB,S_1,S_2)$.
    We arbitrarily orient the new edge $\widehat{e}$.
For each orientable loop $e$ of $\widehat{\bB}$ whose both ends are in $S_2$, we swap the labels $e^+$ and $e^-$ in addition to the fact that $\widehat{\bB}$ is obtained by flipping $S_2$ (Definition~\ref{def: pseudo-orientable bouquet}).
    For the remaining loops in $\widehat{\bB}$, we keep the same orientations as in $\bB$.
    See Figure~\ref{fig: pseudo-orientable bouquet with orientations} for an example.
    Then we obtain the interlacing matrix $\bfM_\pm(\widehat{\bB})$ of $\widehat{\bB}$ with respect to the above orientations.
    Denote 
    \[
        \bfM_\pm(\widehat{\bB}) = 
        \begin{pmatrix}
            \bfA & \bfv \\
            -\bfv^T & 0 \\
        \end{pmatrix},
    \]
    where the last row (or column) is indexed by $\widehat{e}$.
    
    \begin{claim*}
        $\bfM(\bB,S_1,S_2) = \bfA + \bfv \bfv^T$.
    \end{claim*}
    \begin{proof}
        Let $e,f \in E(\bB)$.
        Suppose that $e$ or $f$ is an orientable loop. Then $\bfM(\bB,S_1,S_2)_{ef} = \bfM_\pm(\widehat{\bB})_{ef}$ and $\bfv_e \bfv_f = 0$, so $\bfM(\bB,S_1,S_2)_{ef} = \bfA_{ef} + \bfv_e \bfv_f$.
        Therefore, we may assume that $e$ and $f$ are non-orientable.
        Note that $\bfv_e \bfv_f = 1$.

        If $e=f$ or they interlace, then $\bfA_{ef}=0$ so $\bfM(\bB,S_1,S_2)_{ef} = 1 = \bfA_{ef} + \bfv_e \bfv_f$.

        If $e<f$ and they do not interlace, then $\bfA_{ef}=1$ so $\bfM(\bB,S_1,S_2)_{ef} = 2 = \bfA_{ef} + \bfv_e \bfv_f$.

        If $e>f$ and they do not interlace, then $\bfA_{ef}=-1$ so $\bfM(\bB,S_1,S_2)_{ef} = 0 = \bfA_{ef} + \bfv_e \bfv_f$.
    \end{proof}

    By Lemma~\ref{lem: adjustment of bouquets}, $X$ is a quasi-tree of $\bB$ if and only if $\alpha(X)$ is a quasi-tree of $\widehat{\bB}$. By Theorem~\ref{thm: matrix--quasi-tree theorem for orientable ribbon graphs}, the latter is equivalent to $\det(\bfM_\pm(\widehat{\bB})[\alpha(X)])=1$, and otherwise the determinant is zero. Then we get the desired result by the above claim and Lemma~\ref{lem: determinants of type B and D}. 
\end{proof}

\begin{figure}[h!]
    \centering
    \begin{tikzpicture}
\begin{scope}
            \draw[line width=3pt, red, <-] (50:2) -- (-50:2);
            \draw[line width=3pt, red, <-] (90:2) -- (-130:2);
            \draw[line width=3pt, red, <-] (130:2) -- (-90:2);
            
            \draw[line width=3pt, blue, <-] (30:2) to[bend left=40] (70:2);
            \draw[line width=3pt, blue, ->] (20:2) to[bend left=40] (110:2);

            \draw[line width=3pt, blue, <-] (-165:2) to[bend left=40] (-110:2);
            \draw[line width=3pt, blue, <-] (-150:2) to[bend left=40] (-70:2);
            \draw[line width=3pt, blue, <-] (-10:2) to[bend right=50] (-35:2);

            \draw[line width=2.5pt] (0,0) circle (2);
            
            \draw[line width=1pt, dashed, ->] (5:2.15) arc (5:175:2.15);
            \draw[line width=1pt, dashed, <-] (-5:2.15) arc (-5:-175:2.15);
            \node at (7:2.5) {$S_1$};
            \node at (-10:2.5) {$S_2$};

            \node at (0,-2.6) {$\bB$};
        \end{scope} 

        \begin{scope}[xshift=6cm]
            \draw[line width=3pt, blue, <-] (50:2) -- (-130:2);
            \draw[line width=3pt, blue, <-] (90:2) -- (-50:2);
            \draw[line width=3pt, blue, <-] (130:2) -- (-90:2);

            \draw[line width=3pt, blue, <-] (30:2) to[bend left=40] (70:2);
            \draw[line width=3pt, blue, ->] (20:2) to[bend left=40] (110:2);

            \draw[line width=3pt, blue, ->] (-15:2) to[bend right=40] (-70:2);
            \draw[line width=3pt, blue, ->] (-30:2) to[bend right=40] (-110:2);
            \draw[line width=3pt, blue, ->] (-170:2) to[bend left=50] (-145:2);

\draw[line width=3pt, blue, ->] (0:2) -- (180:2);

            \draw[line width=1pt, dashed, ->] (5:2.15) arc (5:175:2.15);
            \draw[line width=1pt, dashed, ->] (-5:2.15) arc (-5:-175:2.15);
            \node at (7:2.5) {$S_1$};
            \node at (-10:2.5) {$S_2$};

            \draw[line width=2.5pt] (0,0) circle (2);
            
            \node at (0,-2.6) {$\widehat{\bB}$};
        \end{scope} 
    \end{tikzpicture}
    \caption{The left figure shows a pseudo-orientable bouquet $\bB$ with orientations assigned to the loops. The right figure illustrates the corresponding orientations of the loops in its adjustment $\widehat{\bB}$, as described in the proof of Theorem~\ref{thm: matrix--quasi-tree theorem for pseudo-orientable ribbon graphs}. The orientable loops are colored blue, and the non-orientable loops are colored red.
    }
    \label{fig: pseudo-orientable bouquet with orientations}
\end{figure}
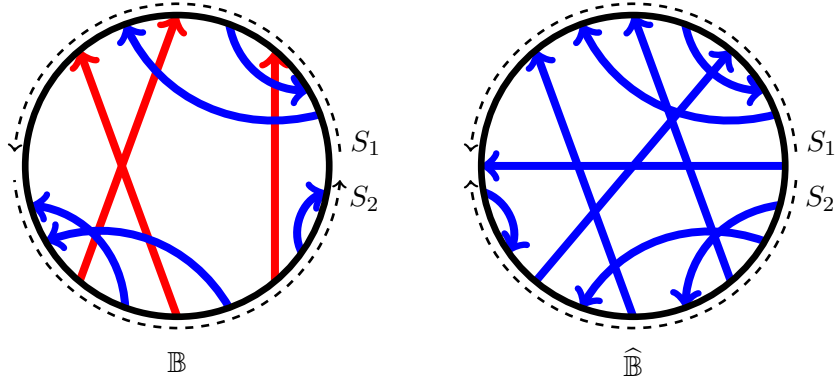

\begin{proof}[\bf Proof of Theorem~\ref{thm-intro: matrix--quasi-tree for pseudo-orientable}]
    Without loss of generality, we may assume that $\bG$ is connected.
    Then $\bG^Q$ is a pseudo-orientable bouquet, so there is a certificate, say $(S_1,S_2)$.
    Then $\bfM = \bfM(\bG^Q,S_1,S_2)$ is the desired matrix by Theorem~\ref{thm: matrix--quasi-tree theorem for pseudo-orientable ribbon graphs} and Proposition~\ref{prop: partial dual and quasi-tree}.
\end{proof}

By Theorem~\ref{thm: matrix--quasi-tree theorem for pseudo-orientable ribbon graphs}, the determinant of $\bfI + {\bfM}(\bB,S_1,S_2)$ is independent of the choice of the certificate $(S_1,S_2)$. However, we remark that its Smith normal form depends on the certificate.

\begin{example}\label{ex: different SNF}
    Let $\bB$ be the pseudo-orientable bouquet in Figure~\ref{fig: different SNF}, which has two non-equivalent certificates $(S_1,S_2)$ and $(T_1,T_2)$.
\begin{figure}[h!]
    \centering
    \begin{tikzpicture}
        \begin{scope}

                \draw[line width=3pt, red] (0+30:2) -- (-90-30:2);
                \draw[line width=3pt, red] (0+60:2) -- (-90-60:2);

                \node at (30:2.3) {$2^+$};
                \node at (-90-30:2.3) {$2^-$};
                \node at (60:2.3) {$3^+$};
                \node at (-90-60:2.3) {$3^-$};

                \draw[line width=3pt, red] (0-30:2) -- (90+30:2);
                \draw[line width=3pt, red] (0-60:2) -- (90+60:2);

                \node at (90+30:2.3) {$5^+$};
                \node at (-30:2.3) {$5^-$};
                \node at (90+60:2.3) {$6^+$};
                \node at (-60:2.3) {$6^-$};

                \draw[line width=3pt, blue] (45:2) to[bend right=75] (15:2);

                \node at (15:2.3) {$1^+$};
                \node at (45:2.3) {$1^-$};

                \draw[line width=3pt, blue] (90+45:2) to[bend right=75] (90+15:2);

                \node at (90+15:2.3) {$4^+$};
                \node at (90+45:2.3) {$4^-$};

                \draw[line width=1pt, dashed, ->] (0+3:2.75) arc (0+3:180-3:2.75);
                \draw[line width=1pt, dashed, ->] (180+3:2.75) arc (180+3:360-3:2.75);

                \node at (90:3) {$S_1$};
                \node at (270:3) {$S_2$};

                \draw[line width=2.5pt] (0,0) circle (2);

\end{scope} 
            \begin{scope}[xshift=7.5cm]

                \draw[line width=3pt, red] (0+30:2) -- (-90-30:2);
                \draw[line width=3pt, red] (0+60:2) -- (-90-60:2);

                \node at (30:2.3) {$2^+$};
                \node at (-90-30:2.3) {$2^-$};
                \node at (60:2.3) {$3^+$};
                \node at (-90-60:2.3) {$3^-$};

                \draw[line width=3pt, red] (0-30:2) -- (90+30:2);
                \draw[line width=3pt, red] (0-60:2) -- (90+60:2);

                \node at (90+30:2.3) {$5^+$};
                \node at (-30:2.3) {$5^-$};
                \node at (90+60:2.3) {$6^+$};
                \node at (-60:2.3) {$6^-$};

                \draw[line width=3pt, blue] (45:2) to[bend right=75] (15:2);

                \node at (15:2.3) {$1^+$};
                \node at (45:2.3) {$1^-$};

                \draw[line width=3pt, blue] (90+45:2) to[bend right=75] (90+15:2);

                \node at (90+15:2.3) {$4^+$};
                \node at (90+45:2.3) {$4^-$};

                \draw[line width=1pt, dashed, ->] (-90+3:2.75) arc (-90+3:90-3:2.75);
                \draw[line width=1pt, dashed, ->] (90+3:2.75) arc (90+3:270-3:2.75);

                \node at (0:3.03) {$T_1$};
                \node at (180:3.03) {$T_2$};

                \draw[line width=2.5pt] (0,0) circle (2);

\end{scope} 
    \end{tikzpicture}
    \caption{A pseudo-orientable bouquet with two certificates $(S_1,S_2)$, left, and $(T_1,T_2)$, right. The red lines indicate non-orientable loops and the blue lines indicate orientable loops.}
    \label{fig: different SNF}
\end{figure}
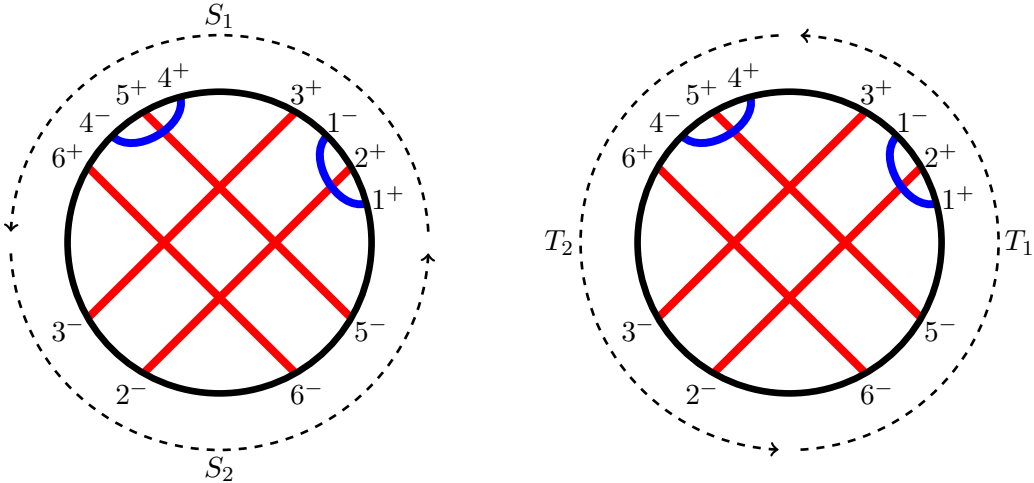
The adjusted interlacing matrices $\bfM(\bB,S_1,S_2)$ and $\bfM(\bB,T_1,T_2)$ are
\[
\left(
    \begin{array}{ccc|ccc}
         0 & 1 & 0 & 0 & 0 & 0 \\
        -1 & 1 & 2 & 0 & 0 & 0 \\
         0 & 0 & 1 & 0 & 0 & 0 \\ \hline
         0 & 0 & 0 & 0 & 1 & 0 \\
         0 & 0 & 0 &-1 & 1 & 2 \\
         0 & 0 & 0 & 0 & 0 & 1 \\
    \end{array}
\right)
\quad \text{and} \quad
\left(
    \begin{array}{ccc|ccc}
         0 & 1 & 0 & 0 & 0 & 0 \\
        -1 & 1 & 2 & 0 & 0 & 0 \\
         0 & 0 & 1 & 0 & 0 & 0 \\ \hline
         0 & 0 & 0 & 0 & 1 & 0 \\
         0 & 0 & 0 &-1 & 1 & 0 \\
         0 & 0 & 0 & 0 & 2 & 1 \\
    \end{array}
\right),
\]
respectively, where the rows (and columns) are indexed by the edges $1,2,3,4,5,6$ in order.
Both matrices have determinant $27$.
The Smith normal form of $\bfI + \bfM(\bB,S_1,S_2)$ is $\mathrm{diag}(1,1,1,1,3,9)$, whereas that of $\bfI + \bfM(\bB,T_1,T_2)$ is $\mathrm{diag}(1,1,1,1,1,27)$.
\end{example}

\section{Stability and log-concavity}\label{sec: stability and log-concavity}

In this section, we study the Hurwitz stability of quasi-tree generating polynomials and log-concavity for regular $\Delta$-matroids and pseudo-orientable ribbon graphs. In particular, we prove Theorems~\ref{thm-intro: stability} and~\ref{thm-intro: log-concavity}.

\subsection{Hurwitz stability of quasi-tree generating polynomials}

We first recall the definition of Hurwitz stability and its basic properties.

\begin{definition}
    A multivariate polynomial $f(z_1,\dots,z_n) \in \bC[z_1,\dots,z_n]$ is \emph{Hurwitz stable} if $f(z_1,\dots,z_n) \ne 0$ whenever $\mathrm{Re}(z_i) > 0$ for all $i$.
\end{definition}

\begin{proposition}[\cite{Wagner2011}]\label{prop: operations preserving Hurwitz stability}
    Suppose $f(z_1,\dots,z_n)$ is Hurwitz stable. Then
    \begin{enumerate}[label=\rm(\roman*)]
        \item\label{item: op-specialization} $f(a,z_2,\dots,z_n)$ is Hurwitz stable for any $a$ with $\mathrm{Re}(a) \ge 0$.
        \item\label{item: op-diagonalization} $f(z_1,z_1,z_3,\dots,z_n)$ is Hurwitz stable.
        \item\label{item: op-inversion} $z_i^d f(z_1,\dots,z_i^{-1},\dots,z_n)$ is Hurwitz stable, where $d$ is the degree of $f$ in $z_i$.

        \item\label{item: op-derivative} $\partial_{z_i}f(z_1,\dots,z_n)$ is Hurwitz stable.
    \end{enumerate}
\end{proposition}

We need the following result from \cite[Lemma~4.1]{Branden2007} to prove Lemma~\ref{lem: skew-symmetric rank-1 perturbation Hurwitz stable}. 
\begin{lemma}[\cite{Branden2007}]
\label{lem: real stable}
Let $A_i$ be complex positive semidefinite $n\times n$ matrices and let $B$ be
complex Hermitian. Then
\[f(z_1,\dots,z_m)=\det(z_1A_1+\cdots+z_mA_m+B)\]
is \emph{stable}. In other words, $f\neq 0$ whenever all the complex numbers $z_i$ are in the upper half-plane.
\end{lemma}

\begin{lemma}\label{lem: skew-symmetric rank-1 perturbation Hurwitz stable}
    Let $\bfA$ be an $n$-by-$n$ skew-symmetric matrix with real entries, and let $\bfv \in \bR^n$.
    Then
    \[
        \sum_{I\subseteq [n]} \det \big( (\bfA + \bfv\bfv^T)[I] \big) z^I
    \]
    is Hurwitz stable, where $z^I := \prod_{i\in I} z_i$.
\end{lemma}
\begin{proof}

By Proposition~\ref{prop: operations preserving Hurwitz stability}\ref{item: op-inversion}, it suffices to show the polynomial
    \[
        \det \big( \bfA + \bfv\bfv^T + \mathrm{diag}(z_1,\dots,z_n) \big) 
=
        \sum_{I\subseteq [n]} \det \big( (\bfA + \bfv\bfv^T)[I] \big) z^{[n]\setminus I}
    \]
    is Hurwitz stable, or equivalently, $\det \big( {\bf i}\bfA + {\bf i}\bfv\bfv^T + \mathrm{diag}(z_1,\dots,z_n) \big)$ is stable, where ${\bf i}$ is the imaginary unit. Since $\bf{i}\bfA$ is complex Hermitian, the above lemma implies that $\det \big( {\bf i}\bfA + z_0\bfv\bfv^T + \mathrm{diag}(z_1,\dots,z_n) \big)$ is stable. Certainly, we can take $z_0={\bf i}$. 
    
   \end{proof}

The \emph{quasi-tree generating polynomial} of a ribbon graph $\bG$ is 
\[
    p_{\bG}(x_e : e\in E(\bG)) := \sum_{Q} x^{Q},
\]
where the sum is over all quasi-trees $Q$ of $\bG$ and $x^Q := \prod_{e\in Q} x_e$.

We prove Theorem~\ref{thm-intro: stability}, that is, the quasi-tree generating polynomial of any pseudo-orientable ribbon graph is Hurwitz stable. It generalizes the same result of Merino, Moffatt, and Noble~\cite{MMN2025b} for orientable ribbon graphs.

\begin{proof}[\bf Proof of Theorem~\ref{thm-intro: stability}]
    Denote $E = E(\bG)$.
    For any quasi-tree $Q$ of $\bG$, we have \[ p_{\bG}(z_i : i\in E) = z^Q \cdot p_{\bG^Q}(z_i,z_j^{-1} : i\in E\setminus Q, \, j\in Q ). \]
    Hence, by Proposition~\ref{prop: operations preserving Hurwitz stability}, we may assume that $\bG$ is a bouquet.

    By Theorem~\ref{thm: matrix--quasi-tree theorem for pseudo-orientable ribbon graphs}, we have 
    \[
        p_\bG = \sum_{Q} \det(\bfM(\bG,S_1,S_2)[Q]) z^Q,
    \]
    where $(S_1,S_2)$ is a certificate of pseudo-orientability for $\bG$. By Lemma~\ref{lem: skew-symmetric rank-1 perturbation Hurwitz stable}, $p_\bG$ is Hurwitz stable.  
    
\end{proof}

In \S\ref{sec: non-detectability of Cn}, we will show that there are infinitely many non-pseudo-orientable ribbon graphs whose quasi-tree generating polynomials are not Hurwitz stable.

\subsection{Log-concavity on quasi-trees of pseudo-orientable ribbon graphs}

Stanley~\cite{Stanley1981} proved that, for a regular matroid, the sequence counting bases with prescribed intersection sizes with a fixed set is log-concave. This implies Mason's conjecture (the Adiprasito--Huh--Katz Theorem~\cite{AHK2018}) for a certain class of matroids.
We extend Stanley's result to regular $\Delta$-matroids. Our proof resembles that of Yan~\cite{Yan2023} generalizing Stanley's result to all matroids by making use of Lorentzian polynomials~\cite{BH2020}.
As a consequence, we obtain a log-concavity result for pseudo-orientable ribbon graphs (Thm.~\ref{thm-intro: log-concavity}).

\begin{definition}
    Let $(a_i)_{i\in \bZ}$ be a sequence of nonnegative real numbers with finite support. 
    Let $\ell := \mathrm{min}_{i}\{a_i \ne 0\}$ and $u := \mathrm{max}_{i}\{a_i \ne 0\}$.
    The sequence has \emph{no internal zero} if $a_i \ne 0$ whenever $\ell \le i \le u$.
    It is \emph{log-concave} if $a_i^2 \ge a_{i-1} a_{i+1}$ for each $i$, and is \emph{ultra-log-concave} if the sequence $\left( a_i / \binom{u-\ell}{i-\ell} \right)_{i\in\bZ}$ is log-concave.
\end{definition}

The \emph{basis generating polynomial} of a $\Delta$-matroid $D = (E,\cB)$ is 
\[
    g_D(x_e : e \in E) := \sum_{B\in \cB} x^B.
\]

\begin{lemma}\label{lem: Hurwitz stability of regular even delta-matroids}
    The basis generating polynomial of any regular $\Delta$-matroid is Hurwitz stable.
\end{lemma}
\begin{proof}
    Since $D$ is regular, there is a PU skew-symmetric matrix $\bfA$ and a subset $X$ of $[n]$ such that $D = D(\bfA) * X$.
    Then, the basis generating polynomial $g_{D}$ is Hurwitz stable by Lemma~\ref{lem: skew-symmetric rank-1 perturbation Hurwitz stable}.
\end{proof}

\begin{theorem}\label{thm: Stanley for regular delta-matroids}
    Let $D$ be a regular $\Delta$-matroid on the set $E$.
    Let $R,S_1,\dots,S_t$ be subsets of $E$, and let $a_1,\ldots,a_t$ be nonnegative integers. 
    For an integer $i$, let $c_i$ be the number of tuples $(R', S_1', \dots, S_t')$ such that 
    \begin{enumerate}[label=\rm(\roman*)]
        \item $R' \in \binom{R}{i}$ and $S_j' \in \binom{S_j}{a_j}$ for all $j$, and
\item $R' \sqcup S_1' \sqcup \dots \sqcup S_t'$ is a partition of a basis of $D$.
    \end{enumerate}
    Then, one of the sequences $(c_{2i})_{i\in\bZ}$ and $(c_{2i+1})_{i\in\bZ}$ is identically zero, and the other is ultra-log-concave with no internal zeros.
\end{theorem}

\begin{proof} Denote $S_0 = R$ and introduce $t+1$ variables $y_0,y_1,\dots,y_t$. We define a polynomial $h_1(y_0,y_1,\dots,y_t)$ from $g_D(x_i : i\in E)$ by substituting each $x_i$ for $\sum_{j: i\in S_j} y_j$. 
    Then $c_i$ is the coefficient of the monomial $y_0^i y_1^{a_1} \dots y_t^{a_t}$ in $h_1$.
    We define a univariate polynomial
    \[
        h_2(z) := 
        \frac{1}{a_1! \cdots a_\ell!} 
        \left(
            \partial_{y_1}^{a_1} \cdots \partial_{y_t}^{a_t} g
        \right)
        (z,0,\dots,0).
    \]
    Then $h_2$ is Hurwitz stable by Proposition~\ref{prop: operations preserving Hurwitz stability}. 
    Note that $c_i$ is the coefficient of $z^{i}$ in $h_2$.

    As $D$ is an even $\Delta$-matroid, all bases have the same parity. 
    
    When every base has even size, $(c_{2i+1})_{i\in\bZ}$ vanishes, and hence $h_2(z) = h_3(z^2)$ for some polynomial $h_3$.
    As $h_2$ is Hurwitz stable and $h_2(z)=h_2(-z)$, the roots of $h_2$ are pure imaginary, so the roots of $h_3$ are negative reals.
    The coefficient of $z^i$ in $h_3$ is $d_i := c_{2i}$.
These coefficients $d_i$ are ultra-log-concave with no internal zeros by Newton's inequality. 

    When every base has odd size, $(c_{2i})_{i\in\bZ}$ vanishes, and hence $h_2(z) = zh_3(z^2)$ for some polynomial $h_3$. By a similar argument, $(c_{2i+1})_{i\in\bZ}$ is ultra-log-concave with no internal zeros.    
\end{proof}

\begin{remark}
Let $M$ be a regular matroid and let $P,Q,S_1,\dots,S_t$ be a partition of the ground set of $M$.
    Let $a_1,\dots,a_t$ be nonnegative integers with $\sum_j a_j \le r(M)$.
    Then the number of bases $B$ of $M$ such that $|B\cap S_j|=a_j$ for all $j$ and $|B\cap P| = i$ (so $|B\cap Q| = r(M) - i - \sum_j a_j$) equals $c_{k+2i}$ in Theorem~\ref{thm: Stanley for regular delta-matroids} applied to $D = M \symdiff Q$ and $R=P\cup Q$, where $k := |Q| - r(M) + \sum_j a_j$.
    Thus, we deduce {\cite[Cor.~2.4]{Stanley1981}} from Theorem~\ref{thm: Stanley for regular delta-matroids}.
\end{remark}

As a consequence of Theorem~\ref{thm: Stanley for regular delta-matroids}, we obtain the following:

\begin{corollary}\label{cor: log-concavity of regular delta-matroids}
    Let $D$ be a regular $\Delta$-matroid and let $\epsilon \in \{0,1\}$ be the parity of bases in $D$.
    If $a_i$ is the number of bases of size $2i+\epsilon$, then the sequence $(a_i)$ is ultra-log-concave.
\end{corollary}
\begin{proof}
    Set $R = E$ and $t=0$ in Theorem~\ref{thm: Stanley for regular delta-matroids}.
\end{proof}

\begin{corollary}\label{cor: log-concavity of quasi-trees of pseudo-orientable ribbon graphs}
    Let $\bG$ be a pseudo-orientable ribbon graph and let $a_i$ (resp. $b_i$) be the number of quasi-trees of size $2i-1$ or $2i$ (resp. $2i$ or $2i+1$).
    Then, the sequences $(a_i)$ and $(b_i)$ are ultra-log-concave.
\end{corollary}
\begin{proof}
    Set $D = D(\widehat{\bG})$ (resp. $D = D(\widehat{\bG}) * \{\widehat{e}\}$), where $\widehat{e}$ is the new edge in~$\widehat{\bG}$.
    Then $D$ is regular by Theorem~\ref{thm: ribbon graphic is binary}.
    By Prop.~\ref{prop: adjustment of ribbon graphs}, the number of quasi-trees of size $2i-1$ or $2i$ (resp. $2i$ or $2i+1$) in $\bG$ equals the number of bases of size $2i$ (resp. $2i+1$) in $D$.
    Thus, the result follows from Corollary~\ref{cor: log-concavity of regular delta-matroids}.
\end{proof}

\begin{proof}[\bf Proof of Theorem~\ref{thm-intro: log-concavity}]
    It follows from Corollary~\ref{cor: log-concavity of quasi-trees of pseudo-orientable ribbon graphs} applied to $\bG^Q$.
\end{proof}

\section{Non-pseudo-orientable ribbon graphs}\label{sec: non-pseudo-orientable ribbon graphs}

We provide an infinite family of non-pseudo-orientable ribbon graphs violating the Matrix--Quasi-tree theorem and Hurwitz stability of quasi-tree generating polynomials.

Let $\bC_n$ be a bouquet consisting of $n$ non-orientable loops, labeled by $1,2,\dots,n$, such that two distinct loops $i$ and $j$ interlace if and only if $|i-j| \equiv 1 \pmod{n}$.
See Figure~\ref{fig: non-pseudo-orientable bouquet - non-Hurwitz} for examples when $n=5,6$.
It is easy to see that $\bC_n$ is not pseudo-orientable if and only if $n\ge 5$.

\begin{figure}[h!]
    \centering
    \begin{tikzpicture}
        \begin{scope}
            \foreach \i in {0,1,2,3,4} {
                \draw[line width=3pt, red] (72*\i+18-15:2) to[bend left=30] (72*\i+90+15:2);
            }
            \draw[line width=2.5pt] (0,0) circle (2);  
            \node at (0,-2.6) {$\bC_5$};  
        \end{scope}
        \begin{scope}[xshift=6cm]
            \foreach \i in {0,1,2,3,4,5} {
                \draw[line width=3pt, red] (60*\i-12:2) to[bend left=35] (60*\i+60+12:2);
            }
            \draw[line width=2.5pt] (0,0) circle (2);    
            \node at (0,-2.6) {$\bC_6$};
        \end{scope}
    \end{tikzpicture}
    \caption{
        $\bC_n$ with $n=5,6$.
        These bouquets are non-pseudo-orientable, and their quasi-tree generating polynomials are not Hurwitz stable.
        }
    \label{fig: non-pseudo-orientable bouquet - non-Hurwitz}
\end{figure}
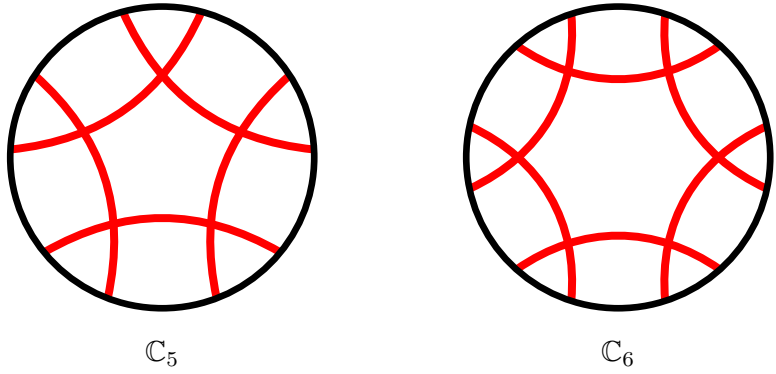

\subsection{The quasi-trees of $\bC_n$ are not detectable}\label{sec: non-detectability of Cn}

We say a real square matrix $\bfM$ \emph{detects} a bouquet $\bB$ if 
\[
    \det(\bfM[X]) =
    \begin{cases}
        1   &   \text{$X$ is a quasi-tree of $\bB$}, \\
        0   &   \text{otherwise}.\\
    \end{cases}
\]

\begin{proposition}\label{prop: no matrix detects Cn}
    For $n\ge 5$, there is no real square matrix $\bfM$ that detects $\bC_n$.
\end{proposition}

\begin{lemma}\label{lem: detectability is minor-closed}
    If a bouquet $\bB$ is detectable by a real square matrix, then so is any minor of $\bB$.
\end{lemma}
\begin{proof}
    Clearly, edge deletion preserves detectability.
    Thus, it suffices to show that partial duality preserves detectability.

    Let $\bfM$ be a real square matrix that detects $\bB$, and let $X$ be a quasi-tree of $\bB$.
    Denote 
    \[
        \bfM = 
        \begin{pmatrix}
            \bfM_1 & \bfM_2 \\
            \bfM_3 & \bfM_4 \\
        \end{pmatrix}
        \quad \text{and} \quad
        \bfM * X =
        \begin{pmatrix}
            \bfM_1^{-1} & -\bfM_1^{-1} \bfM_2 \\
            \bfM_3 \bfM_1^{-1} & \bfM_4 - \bfM_3 \bfM_1^{-1} \bfM_2 \\
        \end{pmatrix},
    \]
    where $\bfM_1$ is the principal submatrix of $\bfM$ indexed by $X$. 
    Then for any $Y \subseteq E(\bB)$,
    \[
        \det((\bfM*X)[Y]) = \det(\bfM[X\symdiff Y]) / \det(\bfM[X])
    \]
    by Tucker's principal pivot transform; see~\cite[Prop.~1]{BH2011}.
    Therefore, $\bfM*X$ detects $\bB^X$.
\end{proof}

\begin{lemma}\label{lem: Cn minor}
    For $n\ge 5$, $\bC_{n-2}$ is a minor of $\bC_{n}$. 
\end{lemma}
\begin{proof}
    $\bC_{n-2}$ is isomorphic to $((\bC_{n})^1)^{\{2,n\}} \setminus \{2,n\}$; see Figure~\ref{fig: Cn minor}.
\end{proof}

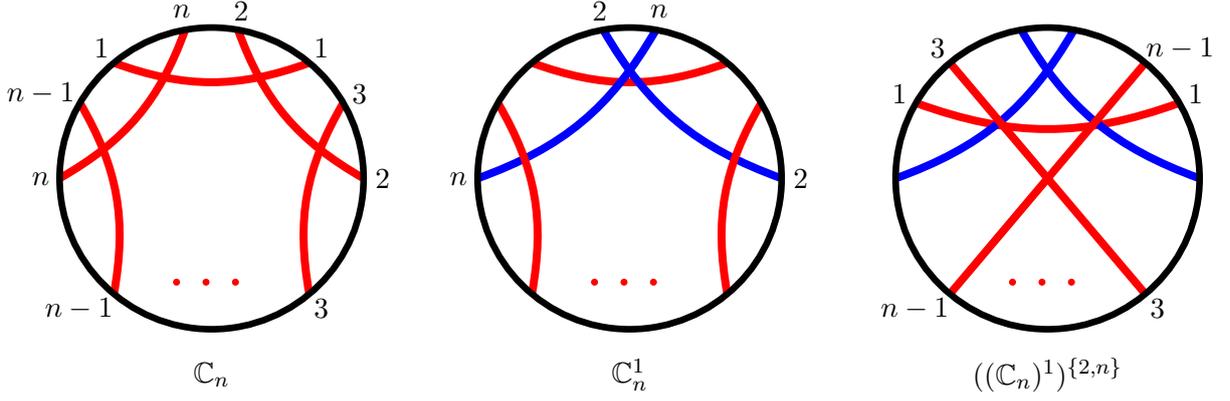
\begin{figure}[h!]
    \centering
    \begin{tikzpicture}
        \begin{scope}
\draw[line width=3pt, red] (90-50*0-40:2) to[bend left=20] (90-50*0+40:2);
            \node at (90-50*0-40:2.25) {$1$};
            \node at (90-50*0+40:2.25) {$1$};
            
\draw[line width=3pt, red] (90-50*1-40:2) to[bend left=20] (90-50*1+40:2);
            \node at (90-50*1-40:2.25) {$2$};
            \node at (90-50*1+40:2.25) {$2$};
\draw[line width=3pt, red] (90+50*1-40:2) to[bend left=20] (90+50*1+40:2);
            \node at (90+50*1-40:2.25) {$n$};
            \node at (90+50*1+40:2.25) {$n$};

\draw[line width=3pt, red] (90-50*2-40:2) to[bend left=20] (90-50*2+40:2);
            \node at (90-50*2-40:2.25) {$3$};
            \node at (90-50*2+40:2.25) {$3$};
\draw[line width=3pt, red] (90+50*2-40:2) to[bend left=20] (90+50*2+40:2);
            \begin{scope}[xshift=-0.3cm]
                \node at (90+50*2-40:2.25) {$n-1$};
                \node at (90+50*2+40:2.25) {$n-1$};
            \end{scope}
    
            \node at (0,-1.4) {\Huge\color{red}$\cdots$};
    
            \draw[line width=2.5pt] (0,0) circle (2);
    
            \node at (0,-2.6) {$\bC_n$};
        \end{scope}
        \begin{scope}[xshift=5.5cm]
\draw[line width=3pt, red] (90-50*0-40:2) to[bend left=20] (90-50*0+40:2);
            
\draw[line width=3pt, blue] (90-50*1-40:2) to[bend left=20] (90+50*1-40:2);
            \node at (90-50*1-40:2.25) {$2$};
            \node at (90+50*1-40:2.25) {$2$};
\draw[line width=3pt, blue] (90-50*1+40:2) to[bend left=20] (90+50*1+40:2);
            \node at (90+50*1+40:2.25) {$n$};
            \node at (90-50*1+40:2.25) {$n$};

\draw[line width=3pt, red] (90-50*2-40:2) to[bend left=20] (90-50*2+40:2);
\draw[line width=3pt, red] (90+50*2-40:2) to[bend left=20] (90+50*2+40:2);
    
            \node at (0,-1.4) {\Huge\color{red}$\cdots$};
    
            \draw[line width=2.5pt] (0,0) circle (2);
    
            \node at (0,-2.6) {$\bC_n^1$};
        \end{scope}
        \begin{scope}[xshift=11cm]
\draw[line width=3pt, red] (90-50*0-40-20:2) to[bend left=20] (90-50*0+40+20:2);
            \node at (90-50*0-40-20:2.25) {$1$};
            \node at (90-50*0+40+20:2.25) {$1$};
            
\draw[line width=3pt, blue] (90-50*1-40:2) to[bend left=20] (90+50*1-40:2);
\draw[line width=3pt, blue] (90-50*1+40:2) to[bend left=20] (90+50*1+40:2);
            
\draw[line width=3pt, red] (90-50*2-40:2) to[bend left=0] (90-50*0+40:2);
            \node at (90-50*2-40:2.25) {$3$};
            \node at (90-50*0+40:2.25) {$3$};
\draw[line width=3pt, red] (90+50*0-40:2) to[bend left=0] (90+50*2+40:2);
            \begin{scope}[xshift=0.3cm]
                \node at (90+50*0-40:2.25) {$n-1$};                
            \end{scope}
            \begin{scope}[xshift=-0.3cm]
                \node at (90+50*2+40:2.25) {$n-1$};
            \end{scope}
    
            \node at (0,-1.4) {\Huge\color{red}$\cdots$};
    
            \draw[line width=2.5pt] (0,0) circle (2);
    
            \node at (0,-2.6) {$((\bC_n)^1)^{\{2,n\}}$};
        \end{scope}
    \end{tikzpicture}
    \caption{The proof of Lemma~\ref{lem: Cn minor}. The red lines represent non-orientable loops, and the blue lines represent orientable loops.}
    \label{fig: Cn minor}
\end{figure}

\begin{lemma}\label{lem: no matrix detects C5}
    There is no real square matrix $\bfM$ that detects $\bC_5$.
\end{lemma}
\begin{proof}
    The quasi-trees of $\bC_5$ are $\emptyset$, $[5]$, and all subsets of $[5]$ of the following forms:
    \begin{enumerate} 
        \item all $\{i\}$ with $i\in[5]$,
        \item all $\{i,i+2\}$ with $i\in[5]$,
        \item all $\{i,i+1,i+2\}$ with $i\in[5]$, and
        \item all $\{i,i+1,i+2,i+3\}$ with $i\in[5]$, 
    \end{enumerate}
    where the addition is modulo $5$.

    Suppose that there is a matrix $\bfM = (a_{ij})_{1\le i,j \le 5}$ that detects $\bC_5$.
    Then $a_{ii} = 1$, $a_{i,i+1} a_{i+1,i} = 1$, and $a_{i,i+2} a_{i+2,i} = 0$ for all $i$, because the above analysis of quasi-trees of size $1$ and $2$.
    Because $\det(\bfM[\{i,i+1,i+2\}]) = 1$, we have that $a_{i,i+2}$ or $a_{i+2,i}$ is nonzero for each $i$.
    
    Taking the transpose if necessary, we may assume that $a_{1,3}$ is nonzero. Then, the conditions $\det(\bfM[\{1,3,4\}]) = \det(\bfM[\{1,2,4\}]) = 0$ imply that $a_{1,4}$ and $a_{2,4}$ are nonzero.
    Similarly, as $a_{2,4}$ is nonzero and $\det(\bfM[\{2,4,5\}]) = \det(\bfM[\{2,3,5\}]) = 0$, we have that $a_{2,5}$ and $a_{3,5}$ are nonzero.
    Thus, the matrix $\bfM$ has the following form:
    \[
        \begin{pmatrix}
            1 & a & * & * & e \\
            1/a & 1 & b & * & * \\
            0 & 1/b & 1 & c & * \\
            0 & 0 & 1/c & 1 & d \\
            1/e & 0 & 0 & 1/d & 1 \\
        \end{pmatrix}
    \]
    where $a,b,c,d,e$ are nonzero real numbers and each $*$ denotes some nonzero real number.
    Then, we deduce $\det(\bfM[\{1,3,5\}]) \ne 0$, which contradicts that $\{1,3,5\}$ is not a quasi-tree of $\bC_5$.
\end{proof}

\begin{lemma}\label{lem: no matrix detects C6}
    There is no real square matrix $\bfM$ that detects $\bC_6$.
\end{lemma}
\begin{proof}
    The quasi-trees of $\bC_6$ are $\emptyset$ and all subsets of $[6]$ of the following forms:
    \begin{enumerate}
        \item all $\{i\}$ with $i\in[6]$,
        \item all $\{i,i+2\}$ with $i\in[6]$, and $\{i,i+3\}$ with $i=1,2,3$,
        \item all $\{i,i+1,i+2\}$ with $i\in[6]$, and $\{1,3,5\}$ and $\{2,4,6\}$,
        \item all four-element subsets of $[6]$ except $[6] \setminus \{i,i+3\}$ with $i=1,2,3$.
    \end{enumerate}
    
    Suppose that there is a matrix $\bfM = (a_{ij})_{1\le i,j\le 6}$ that detects $\bC_6$.
    Then $a_{ii} = 1$, $a_{i,i+1} a_{i+1,i} = 1$, and $a_{i,i+2} a_{i+2,i} = 0$ for all $i$, because the above analysis of quasi-trees of size $1$ and $2$.
    Denote $x_i := a_{i,i+1}$.
    Because $\det(\bfM[\{i,i+1,i+2\}]) = 1$, we have that 
    \begin{itemize}
        \item $a_{i,i+2} = 2 x_i x_{i+1}$ and $a_{i+2,i} = 0$, or
        \item $a_{i,i+2} = 0$ and $a_{i+2,i} = 2 / (x_i x_{i+1})$.
    \end{itemize}

    Taking the transpose if necessary, we may assume that $a_{1,3}$ is nonzero. We now look at the principal submatrix
    \[
        \bfM[\{1,2,3,4\}] =
        \begin{pmatrix}
            1 & x_1 & 2x_1x_2 & a_{1,4} \\
            1/x_1 & 1 & x_2 & a_{2,4} \\
            0 & 1/x_2 & 1 & x_3 \\
            a_{4,1} & a_{4,2} & 1/x_3 & 1 \\
        \end{pmatrix}.
    \]
    Denote $u := a_{1,4}$ and $v := a_{4,1}$.
    Then $0 = \det(\bfM[\{1,3,4\}]) = v (2x_1x_2x_3 - u) = 0$.
    Similarly,
    \[
        0 = \det(\bfM[\{1,2,4\}]) =
        \begin{cases}
            v (2x_1x_2x_3 - u)   &   \text{if $a_{2,4} = x_2x_3$ and $a_{4,2} = 0$}, \\
            u (2/(x_1x_2x_3) - v)   &   \text{if $a_{2,4} = 0$ and $a_{4,2} = 2/(x_2x_3)$}.
        \end{cases}
    \]
    In the latter case, one can check that $\det(\bfM[\{1,2,3,4\}]) = -1$. Therefore, we have $a_{2,4} = x_2x_3$ and $a_{4,2} = 0$.

    By the same argument, we have that $a_{3,5}$, $a_{4,6}$, $a_{5,1}$, and $a_{6,2}$ are nonzero, and $a_{5,3} = a_{6,4} = a_{1,5} = a_{2,6} = 0$. Hence, $\bfM$ has the following form:
    \[
        \begin{pmatrix}
            1 & x_1 & 2x_1x_2 & a_{1,4} & 0 & 1/x_6 \\
            1/x_1 & 1 & x_2 & 2x_2x_3 & a_{2,5} & 0 \\
            0 & 1/x_2 & 1 & x_3 & 2x_3x_4 & a_{6,3}  \\
            a_{4,1} & 0 & 1/x_3 & 1 & x_4 & 2x_4x_5 \\
            2x_5x_6 & a_{5,2} & 0 & 1/x_4 & 1 & x_5 \\
            x_6 & 2x_6x_1 & a_{6,3} & 0 & 1/x_5 & 1 \\
        \end{pmatrix}
    \]
    Then,
    \[
        1 = \det(\bfM[\{1,3,5\}]) = 1 + 8x_1x_2x_3x_4x_5x_6,
    \]
    which contradicts that all $x_i$'s are nonzero.
\end{proof}

\begin{proof}[\bf Proof of Proposition~\ref{prop: no matrix detects Cn}]
    It follows from Lemmas~\ref{lem: detectability is minor-closed}-\ref{lem: no matrix detects C6}.
\end{proof}

\subsection{The quasi-tree generating polynomial of $\bC_n$ is not Hurwitz stable}

\begin{proposition}\label{prop: Cn not Hurwitz}
    The quasi-tree generating polynomial of $\bC_n$ with $n\ge 5$ is not Hurwitz stable.
\end{proposition}

\begin{lemma}\label{lem: C5 C5 not Hurwitz}
    The quasi-tree generating polynomials of $\bC_5$ and $\bC_6$ are not Hurwitz stable.
\end{lemma}
\begin{proof}
    The univariate quasi-tree generating polynomials of $\bC_5$ and $\bC_6$ are
    \begin{equation*}
        x^5 + 5x^4 + 5x^3 + 5x^2 + 5x + 1
        \quad\text{ and }\quad
        12x^4 + 8x^3 + 9x^2 + 6x + 1,
    \end{equation*}
    respectively; see the proofs of Lemmas~\ref{lem: no matrix detects C5} and~\ref{lem: no matrix detects C6}. 
Using a computer, one verifies that each of the two polynomials has a root with positive real part. Therefore, neither of them is Hurwitz stable.
    Consequently, the (multivariate) quasi-tree generating polynomials of $\bC_5$ and $\bC_6$ are not Hurwitz stable.
\end{proof}

We remark that the sequence of the numbers of $i$-sized quasi-trees of $\bC_6$ is not unimodal.

\begin{lemma} \label{lem: quasi-tree generating polynomial and minor}
    Let $\bG$ be a ribbon graph and $e$ be an edge.
    \begin{enumerate}
        \item $p_{\bG^e}(x_e, x_f : f\in E(\bG) \setminus \{e\}) = 
        x_e \cdot p_{\bG}(x_e^{-1}, x_f : f\in E(\bG) \setminus \{e\})$.
        \item $p_{\bG\setminus e}$ is a specialization of $p_{\bG}$ at $x_e = 0$ if $e$ is a coloop of $\bG$
        \item $p_{\bG\setminus e}$ is a specialization of $p_{\bG}$ at $x_e = 1$ if $e$ is not a coloop of $\bG$.
    \end{enumerate}
\end{lemma}
\begin{proof}
    It is straightforward from the definition of quasi-tree generating polynomials and Propositions~\ref{prop: partial dual and quasi-tree} and~\ref{prop: edge deletion and quasi-tree}. 
\end{proof}

\begin{proof}[\bf Proof of Proposition~\ref{prop: Cn not Hurwitz}]
    It follows from Proposition~\ref{prop: operations preserving Hurwitz stability} and Lemmas~\ref{lem: C5 C5 not Hurwitz}, \ref{lem: Cn minor}, and \ref{lem: quasi-tree generating polynomial and minor}.
\end{proof}

Finally, we remark that there is a non-pseudo-orientable ribbon graph whose quasi-tree generating polynomial is Hurwitz stable.

\begin{proposition}\label{prop: non-pseudo Hurwitz}
    Let $\bG$ be a bouquet depicted in Figure~\ref{fig: non-pseudo-orientable bouquet - Hurwitz}.
    Then $\bG$ is not pseudo-orientable and $p_{\bG}$ is Hurwitz stable.
\end{proposition}

\begin{figure}[h!]
    \centering
    \begin{tikzpicture}
        \foreach \i in {0,1,2,...,13} {
            \coordinate (\i) at (360/14 * \i : 2);
        }
        
        \draw[line width=3pt, red] (1) to[bend left=10] (6);
        \draw[line width=3pt, red] (2) -- (9);
        \draw[line width=3pt, red] (5) -- (12);

        \draw[line width=3pt, blue] (0) to[bend left=20] (3);
        \draw[line width=3pt, blue] (4) to[bend left=20] (7);
        \draw[line width=3pt, blue] (8) to[bend left=20] (11);
        \draw[line width=3pt, blue] (10) to[bend left=20] (13);

        \node at (360/14 * 1:2.25) {$1$};
        \node at (360/14 * 0:2.25) {$2$};
        \node at (360/14 * -2:2.25) {$3$};
        \node at (360/14 * -3:2.25) {$4$};
        \node at (360/14 * -4:2.25) {$5$};
        \node at (360/14 * -5:2.25) {$6$};
        \node at (360/14 * -7:2.25) {$7$};

        \draw[line width=2.5pt] (0,0) circle (2);    
    \end{tikzpicture}
    \caption{A non-pseudo-orientable bouquet whose quasi-tree generating polynomial is Hurwitz stable. The red loops are non-orientable, and the blue loops are orientable.}
    \label{fig: non-pseudo-orientable bouquet - Hurwitz}
\end{figure}
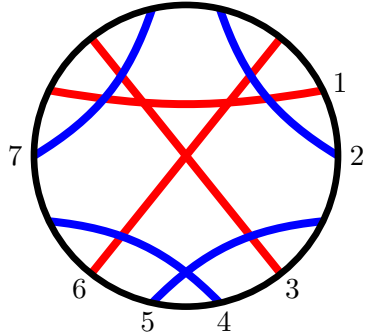

\begin{proof}
    It is easily seen that $\bG$ is not pseudo-orientable.

    We claim that $p_{\bG}$ is Hurwitz stable.
    Denote $D:=D(\bG)$.
    It is straightforward to check that the interlacing matrix $\bfM_2(\bG)$ equals the binary adjacency matrix of the left graph in Figure~\ref{fig: non-pseudo-orientable bouquet - Hurwitz2} plus the diagonal matrix $\mathrm{diag}(1,0,1,0,0,1,0)$.
    Also, $\widehat{\bfM_2(\bG)}$ equals the binary adjacency matrix of the underlying graph of the right digraph in Figure~\ref{fig: non-pseudo-orientable bouquet - Hurwitz2}.
The binary $\Delta$-matroid $\widehat{D} := D(\widehat{\bfM_2(\bG)})$ is not ribbon-graphic by {\cite[Thm.~1.2]{GO2009}}. Then one can check by direct computation that the real adjacency matrix of the right digraph in Figure~\ref{fig: non-pseudo-orientable bouquet - Hurwitz2} is principally unimodular and is a representation of $\widehat{D}$. Thus, $\widehat{D}$ is regular.
    By Lemma~\ref{lem: Hurwitz stability of regular even delta-matroids}, $g_{\widehat{D}}$ is Hurwitz stable. 
    Note that $p_{\bG} = g_{D}$, and $g_{D}$ is a specialization of $g_{\widehat{D}}$ at $x_8 = 1$.
    Therefore, $p_{\bG}$ is Hurwitz stable by Proposition~\ref{prop: operations preserving Hurwitz stability}.
\end{proof}

\begin{figure}[h!]
    \centering
    \begin{tikzpicture}
        \begin{scope}
            \foreach \i in {1,2,3,...,7} {
                \node at (90 + 360/7 - 360/7 * \i : 1.75) {$\i$};
                \coordinate (\i) at (90 + 360/7 - 360/7 * \i : 1.5);
}

            \draw[line width=2pt]
                (1)--(2)--(3)--(4)--(5)--(6)--(7)--cycle
                                (1)--(3)
                                (3)--(6)
                                (6)--(1);
        \end{scope}

        \begin{scope}[xshift=5cm]
            \foreach \i in {1,2,3,...,7} {
                \node at (90 + 360/7 - 360/7 * \i : 1.75) {$\i$};
                \coordinate (\i) at (90 + 360/7 - 360/7 * \i : 1.5);
                \draw[line width=2pt, stealth-] (\i) -- (90 - 360/7 * \i : 1.5);
            }

            \coordinate (8) at (0,0);
            \node at (0,-0.25) {$8$};

            \draw[line width=2pt, -stealth] (1)--(8);
            \draw[line width=2pt, stealth-] (3)--(8); 
            \draw[line width=2pt, stealth-] (6)--(8);
        \end{scope}
    \end{tikzpicture}
    \caption{The binary adjacency matrix of the left graph is a binary representation of $D$ in the proof of Proposition~\ref{prop: non-pseudo Hurwitz}. The real adjacency matrix of the right digraph is a regular representation of $\widehat{D}$ in the same proof.}
    \label{fig: non-pseudo-orientable bouquet - Hurwitz2}
\end{figure}
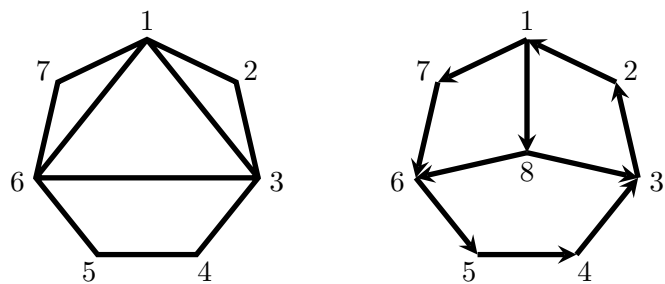

\subsection{Failure of ultra-log-concavity}\label{sec: failure of ultra-log-concavity}

The ultra-log-concavity property does not hold for general ribbon graphs and hence strong $\Delta$-matroids. Still consider the ribbon graphs $\bC_5$ and $\bC_6$.
Let $q_i$ denote the number of quasi-trees of size $2i$ or $2i+1$ of $\bC_5$. Then $q_0 = 6$, $q_1 = 10$, $q_2 = 6$, and $q_j=0$ for $j\ge 3$; see the proof of Lemma~\ref{lem: no matrix detects C5}.
Let $r_i$ denote the number of quasi-trees of size $2i$ or $2i+1$ of $\bC_6$. Then $r_0 = 7$, $r_1 = 17$, $r_2 = 12$, and $r_j=0$ for $j\ge 3$; see the proof of Lemma~\ref{lem: no matrix detects C6}.
Then, we have
\[
    q_1^2 = 100 < 144 = 4 q_0 q_2
    \quad\text{and}\quad
    r_1^2 = 289 < 336 = 4 r_0 r_2,
\]
i.e., the sequences $(q_0,q_1,q_2,\dots)$ and $(r_0,r_1,r_2,\dots)$ are not ultra-log-concave.

We also note that there is an infinite family of even $\Delta$-matroids that fail to satisfy ultra-log-concavity.
\begin{example}
    For $n \ge 5$, let $D$ be a set system $([n],\cB)$ where $\cB$ is the set of all subsets of $[n]$ of size $1$, $3$, or $5$.
    It is readily seen that $D$ is an even $\Delta$-matroid. Denote $b_i$ the number of bases of size $2i+1$ of $D$. Then $b_0 = n$, $b_1 = \binom{n}{3}$, and $b_2 = \binom{n}{5}$. The ratio $\frac{b_1^2}{b_0 b_2}$ converges to $\frac{5!}{(3!)^2} = \frac{120}{36} = \frac{10}{3} < 4$ as $n \to \infty$. Therefore, the sequence $(b_0,b_1,b_2)$ is not ultra-log-concave for sufficiently large $n$.
\end{example}

Finally, we ask the following question.

\begin{question}
    Let $D$ be an even $\Delta$-matroid, and let $b_i$ denote the number of bases of size $2i$ or $2i+1$ of $D$. Is the sequence $(b_0,b_1,b_2,\dots)$ log-concave?
\end{question}

\section*{Acknowledgements}
The authors thank Matt Baker and Mark Ellingham for helpful discussions. In particular, Mark Ellingham suggested Lemma~\ref{lem: simple criterion for pseudo-orientability} and simplified the proof of Proposition~\ref{prop: partial duality preserves pseudo-orientability}. The authors also thank V{\'a}clav Rozho{\v{n}} and Robert \v{S}{\'a}mal for sharing a counterexample to the ultra-log-concavity question posed in an earlier version of this paper.

\bibliographystyle{plainurl}

\appendix

\end{document}